\documentclass[11pt,a4paper]{article}
\usepackage[T1]{fontenc}
\usepackage[margin=28mm]{geometry}
\usepackage{amsmath,amssymb,amsthm,booktabs,array}
\usepackage[colorlinks=true,linkcolor=blue,citecolor=blue,urlcolor=blue]{hyperref}
\newtheorem{theorem}{Theorem}[section]
\newtheorem{lemma}[theorem]{Lemma}
\newtheorem{proposition}[theorem]{Proposition}
\newtheorem{corollary}[theorem]{Corollary}
\theoremstyle{remark}

\newcommand{\cl}[1]{\overline{#1}}
\newcommand{\bG}{\overline G}
\newcommand{\Hc}{\mathcal H}
\title{Hamilton-connected cores and five cycle--wheel Ramsey numbers}
\author{Zehui Shao, Hanxin Jiang\\[0.5ex]
\small Institute of Computing Science and Technology,\\
\small Guangzhou University, Guangzhou 510006, China}
\date{}
\begin{document}
\maketitle
\begin{abstract}
Let $W_s=K_1+C_{s-1}$ denote the wheel on $s$ vertices.
We give structural proofs that $R(C_{14},W_{11})=27$ and
$R(C_{15},W_{11})=29$.
Together with the theorem of Chen et al. for $n\ge16$,
these equalities give $R(C_n,W_{11})=2n-1$ for every $n\ge14$.
The two boundary values were included in an earlier survey announcement.
We also give structural proofs of $R(C_8,W_7)=15$,
$R(C_9,W_7)=17$, and $R(C_8,W_9)=15$.
The common starting point is a Hamilton-connected core lemma.
For the eleven-vertex wheel, bounds on vertex connectivity and on the
matching number of a bipartite graph associated with a local cycle
yield a vertex cut of order nine. Paths with prescribed endpoints
then rule out every possible pair of orders of the two remaining
vertex sets.
For the smaller wheels, we use the structure of critical cycle
colorings and local cycle-shortening arguments. We also give complete
structural classifications of the $(C_8,C_6)$- and $(C_9,C_6)$-critical
colorings, recovering the previously reported counts 24 and 26.
All proofs are combinatorial and use no exhaustive graph enumeration.
\end{abstract}

\section{Introduction}
For graphs $F$ and $J$, the Ramsey number $R(F,J)$ is the least integer
$N$ such that every red--blue edge coloring of $K_N$ contains a red
copy of $F$ or a blue copy of $J$. All graphs are finite and simple,
and copies need not be induced. We write $G$ for the red graph and
$\cl G$ for the blue graph. Throughout, $W_s=K_1+C_{s-1}$ denotes
the wheel on $s$ vertices, following Radziszowski's
survey~\cite{survey18}. Thus our $W_s$ is denoted by $W_{s-1}$ in
papers that index wheels by their rim length.

For an odd-order wheel, the disjoint union of two copies of $K_{n-1}$
gives the natural lower bound $R(C_n,W_{2k+1})\ge2n-1$:
the red graph contains no $C_n$, and its blue complement is bipartite
and contains no wheel. Chen et al.~\cite{chen2009}
proved that
\[
 R(C_n,K_1+C_{2k})=2n-1,
 \qquad k\ge2,\quad n\ge3k+1.
\]
For $W_7$, $W_9$, and $W_{11}$, their theorem gives the respective
ranges $n\ge10$, $n\ge13$, and $n\ge16$.
We give structural proofs of the two cases immediately preceding
the last range.

\begin{theorem}\label{thm:eleven}
\[
 R(C_{14},W_{11})=27,\qquad R(C_{15},W_{11})=29.
\]
\end{theorem}

\begin{corollary}\label{cor:eleven-range}
For every integer $n\ge14$,
\[
 R(C_n,W_{11})=2n-1.
\]
\end{corollary}
\begin{proof}
The two boundary cases are Theorem~\ref{thm:eleven}, and the cases
$n\ge16$ follow from the theorem of Chen et al.~\cite{chen2009}.
\end{proof}

The two equalities in Theorem~\ref{thm:eleven} were included in a
broader formula reported in the 2009 revision of Radziszowski's
survey~\cite[Section~3.4(e)]{survey12}. That revision recorded
$R(C_n,W_s)=2n-1$ for odd $s$ under the condition $2n\ge3s-5$,
attributing the result to the then-forthcoming paper of Chen et al.
For $s=11$, this condition gives $n\ge14$.
The published theorem~\cite{chen2009} establishes the range
$2n\ge3s-1$, and this is also the range recorded in the 2011 and
2026 survey revisions~\cite{survey13,survey18}. Thus the formula
stated there for $W_{11}$ starts at $n=16$ and does not cover the
two boundary cases $n=14,15$. We have not located published proofs
of these two cases. The present paper provides complete structural
proofs, extending the range covered by the published theorem to
$n\ge14$. The numerical values themselves were already implicit
in the 2009 survey report.

The proof of Theorem~\ref{thm:eleven} follows the same approach in
both cases. Suppose that a coloring of $K_{2n-1}$ has no red $C_n$
and no blue $W_{11}$. We first determine the minimum degree of the
red graph $G$ and obtain a lower bound on its vertex connectivity.
Choose a vertex $v$ of minimum red degree and put $X=N_G(v)$.
The blue neighborhood of $v$ contains a long red cycle $C$.
Restrictions on paths with endpoints on $C$ bound the matching
number $\nu(G[V(C),X])$ of the red bipartite graph with parts
$V(C)$ and $X$. By K\"onig's theorem, this bipartite graph has a
vertex cover of the corresponding size. This cover, enlarged if
necessary, together with the vertices of the blue neighborhood
outside $C$, yields a vertex cut of order nine. The absence of a
blue wheel gives minimum-degree bounds within the two remaining
vertex sets. We then construct red paths with suitable endpoints
in these sets and join them through the cut to obtain a red $C_n$.
The relevant parameters are summarized below.
\[
\begin{array}{c|c|c|c|c|c}
n&|G|&\delta(G)&\kappa(G)&\text{local cycle}
 &\nu(G[V(C),X])\\\hline
14&27&9&\ge8&C_{12}&\le4\\
15&29&10&\ge9&C_{14}&5
\end{array}
\]
The connectivity reductions are proved in
Section~\ref{sec:w11-connectivity}; the local matching arguments
and the completion of the two cases are given in
Sections~\ref{sec:w11-fourteen} and~\ref{sec:w11-fifteen}.

Our second result concerns smaller wheels.
\begin{theorem}\label{thm:main}
\[
 R(C_8,W_7)=15,\qquad R(C_9,W_7)=17,\qquad R(C_8,W_9)=15.
\]
\end{theorem}

The numerical history of the first two equalities requires care.
The 2009 revision of the survey~\cite[Table~VI]{survey12} already
reported $R(C_8,W_7)=15$ and the formula $R(C_n,W_7)=2n-1$ for
$n\ge8$, attributing them to then-unpublished work of Chen et al. The 2011 and 2026 revisions~\cite{survey13,survey18}
give the range $n\ge10$, in agreement with the published
theorem~\cite{chen2009}; see also Wang and
Zhang~\cite[Section~1 and Lemma~2.1]{wangzhang2025}.
Our contribution to these two equalities is therefore their
structural proofs, and we do not claim priority for the numerical
values. For $R(C_8,W_9)$, Table~VI of the 2026
survey~\cite{survey18} has a blank entry, while the bound of
Zhang and Chen~\cite[Lemma~3.7]{zc2026} gives
$R(C_8,W_9)\le18$. The same lemma gives only
$R(C_{14},W_{11})\le31$ and does not apply to the odd cycle $C_{15}$.

\begin{corollary}\label{cor:range}
For every integer $n\ge6$,
\[
 R(C_n,W_7)=2n-1.
\]
\end{corollary}
\begin{proof}
The values at $n=6,7$ are reported in the survey~\cite[Table~VI]{survey18},
with attribution to the computations of Luo et al.~\cite{luoll}.
Luo et al. index wheels by rim length: their $W_6$ is the graph
$K_1+C_6$, denoted by $W_7$ here. Their values 11 and 13 therefore
give the two cases in our notation.
The cases $n=8,9$ follow from Theorem~\ref{thm:main}, and those
with $n\ge10$ follow from Chen et al.~\cite{chen2009}.
\end{proof}

The common structural tool is Lemma~\ref{lem:core}: in a graph on
$2n-1$ vertices, a Hamilton-connected subgraph on $n-1$ vertices
forces a $C_n$ or a wheel in the complement, provided the parameter
conditions of the lemma hold. We formulate it as a common auxiliary
reduction based on spanning paths and wheel constructions; its relation
to earlier arguments is discussed in Section~\ref{sec:prelim}.
For $W_7$, suitable cores occur in the critical cycle colorings.
For $n\in\{8,9\}$, let $\Hc_n$ consist of the isomorphism classes
of graphs $H$ of order $n+1$ with $C_n\not\subseteq H$ and
$C_6\not\subseteq\cl H$. Since $R(C_8,C_6)=10$ and
$R(C_9,C_6)=11$, these are the critical colorings of order one less
than the corresponding Ramsey number; no edge-maximality is required,
and isomorphisms preserve the color roles. We prove the core
descriptions needed for the Ramsey argument in
Section~\ref{sec:wheel-seven}. Appendix~\ref{sec:classification}
gives the complete classifications, recovering structurally the
counts $|\Hc_8|=24$ and $|\Hc_9|=26$ previously reported by
Dzido~\cite[proof of Theorem~20]{dzido}.

For $W_9$, a local lemma shows that every red--blue coloring of
$K_{10}$ without a monochromatic $C_8$ contains a monochromatic
Hamilton-connected subgraph on seven vertices. Together with
cycle-shortening lemmas, this proves the third equality of
Theorem~\ref{thm:main} in Section~\ref{sec:wheel-nine}.
The complete arguments use explicit path constructions; no step
depends on a computer enumeration of graphs.

\section{Preliminaries and a core lemma}\label{sec:prelim}
For a graph $F$, we write $|F|$ for its order, $e(F)$ for its number
of edges, and $F[U]$ for the subgraph induced by $U\subseteq V(F)$.
For disjoint vertex sets $A,B\subseteq V(F)$, let $F[A,B]$ denote
the bipartite graph with parts $A$ and $B$ and all edges of $F$
joining the two parts. A matching is a set of pairwise
vertex-disjoint edges. The \emph{matching number} $\nu(F)$ is the
maximum size of a matching in $F$. A matching \emph{saturates} a vertex set if every
vertex of that set is incident with an edge of the matching.
A vertex cover is a set of vertices meeting every edge.
We write $\kappa(F)$ for the vertex connectivity of $F$,
$H_j$ for an arbitrary graph on $j$ vertices, and $\vee$ for the join.

A graph $G$ is \emph{Hamilton-connected} if every two distinct
vertices are the endpoints of a spanning path. We specify paths
by their orders; thus $P_r$ has $r$ vertices and $r-1$ edges.
The girth and circumference of $G$ are denoted by $g(G)$ and $c(G)$.
A graph $G$ is \emph{weakly pancyclic} if it contains a cycle of
every length from $g(G)$ to $c(G)$.

Dirac's theorem states that a graph of order $r\ge3$ with minimum
degree at least $r/2$ is Hamiltonian, and his circumference theorem states that a
2-connected graph satisfies $c(G)\ge\min\{2\delta(G),|G|\}$~\cite{dirac}.
The theorem of Brandt et al.~\cite{bfg} states that a
nonbipartite graph $G$ with $\delta(G)\ge(|G|+2)/3$ is weakly pancyclic
and has girth three or four.
For the cycle Ramsey numbers and the elementary monochromatic-cycle
lemmas used below, we refer to K\'arolyi and Rosta~\cite{kr}.

We recall a standard degree condition for Hamilton-connectedness.
\begin{lemma}\label{lem:hc}
  If $Q$ has $r\ge3$ vertices and $\delta(Q)\ge(r+1)/2$, then $Q$ is
  Hamilton-connected.
\end{lemma}
\begin{proof}
  We use the elementary closure observation that if $u,v$ are nonadjacent
  vertices of an $s$-vertex graph $F$ and $d(u)+d(v)\ge s$, then $F+uv$ is
  Hamiltonian if and only if $F$ is Hamiltonian.
  Indeed, a Hamilton cycle using $uv$ yields a spanning path
  $x_1\cdots x_s$ in $F$, where $x_1=u$ and $x_s=v$.
  Among the $s-1$ indices $1\le i\le s-1$, the conditions
  $ux_{i+1}\in E(F)$ and $vx_i\in E(F)$ hold on index sets whose
  cardinalities sum to at least $s$.
  An index satisfying both gives the Hamilton cycle
  $x_1\cdots x_i x_s x_{s-1}\cdots x_{i+1}x_1$ in $F$.

  Fix distinct $a,b\in V(Q)$ and add a new vertex $w$ adjacent only to $a,b$.
  Every nonadjacent pair of original vertices has degree sum at least $r+1$.
  The closure step allows all missing edges among the original vertices to
  be added without changing Hamiltonicity. The resulting graph is a clique
  on the original vertices together with $w$, and is Hamiltonian.
  Deleting $w$ from a Hamilton cycle in the original augmented graph gives
  the required spanning $a$--$b$ path in $Q$.
\end{proof}

\paragraph{Motivation and related constructions.}
The motivation for the next lemma is to combine two familiar
constructions: closing a spanning path through an outside vertex,
and extending a star to a wheel using edges between two vertex sets.
In the low-connectivity part of their proof of Theorem~4,
Zhang, Broersma and Chen~\cite[Section~3.1]{zbc2015} use
Hamilton-connectedness on $m-1$ vertices to obtain a spanning path
between two neighbors of an outside vertex, and hence a cycle $C_m$.
The same argument also constructs a wheel from a star and complete
bipartite connections between two parts.
Chng et al.~\cite[Observation~4.3 and Lemma~4.4]{trees} give related
constructions for a disjoint union: a star in the complement of one
part, together with the complete cross-edges in the complement, yields
a wheel or a structural restriction on the other part.
Their wheel $W_8=K_1+C_8$ is denoted by $W_9$ here.

The lemma below assumes Hamilton-connectedness directly on a set $X$
of $n-1$ vertices in a graph of order $2n-1$. The core may have missing
edges and need not be a component. If a red $C_n$ is forbidden,
each outside vertex has at most one red neighbor in $X$: two such
neighbors would be the endpoints of a spanning path that closes
through that vertex. Thus the blue edges between the two parts may
be incomplete, with at most one missing edge at each outside vertex.
This is the distinction from the complete cross-connections in the
constructions cited above. Hall's theorem selects distinct connectors
for an alternating blue rim, and Dirac's theorem handles the case
where no outside vertex has enough blue neighbors outside $X$.
We record the resulting cycle-or-wheel alternative for general
$K_1+C_{2k}$ in the explicit range below, with a self-contained proof.

The two numerical conditions in the next lemma serve different steps
of this construction. The inequality $n\ge k+4$ leaves at least
$n-4\ge k$ choices in the core for each rim connector, after imposing
blue adjacency to the hub and the two neighboring rim vertices.
If no hub has $k$ blue neighbors outside the core, the red graph on
the $n$ outside vertices has minimum degree at least $n-k$;
the inequality $n\ge2k$ then permits an application of Dirac's theorem.
\begin{lemma}\label{lem:core}
  Let $k\ge2$ and $n\ge\max\{2k,k+4\}$. If a graph $G$ of order $2n-1$ contains a
  Hamilton-connected subgraph of order $n-1$, then $G$ contains $C_n$ or
  $\cl G$ contains $K_1+C_{2k}$.
\end{lemma}
\begin{proof}
  Let $X$ be the vertex set of the Hamilton-connected subgraph and put
  $T=V(G)\setminus X$, so $|T|=n$. Suppose $G$ contains no $C_n$.
  Each $t\in T$ has at most one neighbor in $X$, since two such neighbors
  can be joined by a spanning path in $X$ and closed through $t$.

  If some $t\in T$ has at least $k$ blue neighbors in $T$, choose
  $t_1,\ldots,t_k$ among them and put
  $X_0=X\cap N_{\cl G}(t)$, so $|X_0|\ge n-2$.
  For $1\le i\le k$, with $t_{k+1}=t_1$, set
  \[
  S_i=X_0\cap N_{\cl G}(t_i)\cap N_{\cl G}(t_{i+1}).
  \]
  Each $t_i$ has at most one red neighbor in $X$, so
  $|S_i|\ge n-4\ge k$.
  These $k$ sets satisfy Hall's condition. Choose distinct representatives
  $x_i\in S_i$. The blue cycle
  $t_1x_1t_2x_2\cdots t_kx_kt_1$, together with its blue hub $t$, forms
  $K_1+C_{2k}$.
  Otherwise $\Delta(\cl G[T])\le k-1$, so
  $\delta(G[T])\ge n-k\ge n/2$. Dirac's theorem gives a red $C_n$ in $T$.
\end{proof}

\paragraph{Role in the proofs.}
Lemma~\ref{lem:core} excludes a Hamilton-connected subgraph of order
$n-1$ from any putative $(C_n,W_{2k+1})$ counterexample of order
$2n-1$ in its parameter range. For $W_{11}$, the cases
$(n,k)=(14,5),(15,5)$ rule out dense components and local
configurations in the connectivity reductions. For $W_7$, the cases
$(n,k)=(8,3),(9,3)$ apply to the cores supplied by the critical-cycle
analysis in Section~\ref{sec:wheel-seven}; the two bipartite exceptions
at $n=9$ are treated separately there. For $W_9$, the case
$(n,k)=(8,4)$ gives Corollary~\ref{w9:cor:no-core}, which starts the
structural reductions in Section~\ref{sec:wheel-nine}.
The lemma therefore lets the later arguments focus on finding a core
in a component or a local configuration, while using a single proof
to obtain the required cycle-or-wheel contradiction.

We also use the following restriction on attachments to a cycle.
\begin{lemma}\label{lem:successors}
  Let $C=x_0x_1\cdots x_{r-1}x_0$ be a red cycle, and let $t\notin V(C)$.
  If there is no red $C_{r+1}$, then the red neighbors of $t$ on $C$ form
  an independent set in $C$. Moreover, the successors of these neighbors
  on $C$ form a blue clique, and so do their predecessors.
  Indices are taken modulo $r$.
\end{lemma}
\begin{proof}
  Two consecutive red neighbors allow $t$ to be inserted into $C$.
  If $tx_i,tx_j$ and $x_{i+1}x_{j+1}$ were all red, traverse $C$ from
  $x_i$ backwards to $x_{j+1}$, use the chord to $x_{i+1}$, and then
  traverse forwards to $x_j$. This is a red spanning path on $V(C)$;
  closing it through $t$ gives $C_{r+1}$. The predecessor statement follows
  by reversing the orientation of $C$.
\end{proof}

Bondy's theorem states that a graph of order $q\ge3$ with minimum
degree at least $q/2$ is pancyclic unless it is $K_{q/2,q/2}$~\cite{Bondy}.
We use the following sufficient condition for Hamiltonicity from
Chv\'atal's degree-sequence theorem~\cite{Chvatal}: a graph of order
$q\ge3$ with degree sequence $d_1\le\cdots\le d_q$ is Hamiltonian
if $d_i>i$ for every integer $1\le i<q/2$. We also use Hall's theorem and K\"onig's theorem, which states that
the matching number of a bipartite graph equals the minimum size
of a vertex cover. Williamson's panconnectedness theorem~\cite{Williamson}
states that minimum degree at least $(q+2)/2$ guarantees, between
any two distinct vertices, a path of every order from their distance
plus one to $q$.
The cycle--cycle Ramsey formula~\cite{FSR} gives
\[
R(C_{12},C_{10})=16,\qquad R(C_{14},C_{10})=18,
\qquad R(C_{15},C_{10})=19.
\]

\section{Preliminary structure for the eleven-vertex wheel}
\label{sec:w11-setup}
In Sections~\ref{sec:w11-setup}--\ref{sec:w11-fifteen}, a
\emph{counterexample at $n$}, for $n\in\{14,15\}$, means a red
graph $G$ on $2n-1$ vertices with no $C_n$ and with no
$W_{11}=K_1+C_{10}$ in its blue complement. All degrees and vertex cuts refer to $G$, unless a color or
another graph is specified. The lemmas stated for arbitrary sets or graphs are
independent of the choice of $n$.

\subsection{Degree and blue separation}
\begin{lemma}\label{lem:blue-sides}
In a coloring with no blue $W_{11}$, if $A,B$ are disjoint sets with all cross-edges blue and $|B|\ge5$,
then each vertex of $A$ has at most four blue neighbors within $A$.
Consequently $\delta(G[A])\ge |A|-5$.
\end{lemma}
\begin{proof}
Five blue neighbors of a vertex $a$ in $A$, alternating with five
vertices of $B$, would form a blue $C_{10}$ with blue hub $a$.
\end{proof}

\begin{corollary}
\label{cor:small14}
For a counterexample at $n=14$, the red graph $G$ is
nonbipartite, $2$-connected, and has
minimum degree exactly nine.
\end{corollary}
\begin{proof}
Every blue neighborhood has no blue $C_{10}$ and no red $C_{14}$.
Since $R(C_{14},C_{10})=18$~\cite{FSR}, it has at most seventeen
vertices. Thus $\delta(G)\ge9$. A bipartition of $G$ would have
an independent part of at least fourteen vertices, whose blue
clique contains a wheel. Hence $G$ is nonbipartite.

If $G$ were disconnected, every component would have at least ten
vertices. There cannot be three components, since a hub in one
and five vertices from each of two others give a blue wheel.
With two components, the larger has order $a\ge14$ and minimum
red degree at least $a-5>a/2$, by Lemma~\ref{lem:blue-sides}.
Bondy's theorem~\cite{Bondy} supplies a red $C_{14}$.
If $u$ were a cut vertex, every component of $G-u$ would have
at least nine vertices. Again there would be two components.
A component of order at least fourteen gives the same contradiction.
Otherwise both components have thirteen vertices and internal
minimum red degree at least eight, so either is a
Hamilton-connected core, contrary to Lemma~\ref{lem:core}.
Therefore $G$ is $2$-connected.

If $\delta(G)\ge10$, the theorem of Brandt et al.~\cite{bfg} applies because $10\ge(27+2)/3$ and $G$ is
nonbipartite. It gives weak pancyclicity and girth three or four.
Dirac's circumference bound~\cite{dirac} is at least twenty,
so $G$ contains a red $C_{14}$. Thus $\delta(G)=9$.
\end{proof}

\begin{corollary}\label{cor:small}
For a counterexample at $n=15$, $G$ is nonbipartite,
$2$-connected, and $\delta(G)=10$.
\end{corollary}
\begin{proof}
A bipartition of $G$ would have an independent part of order at
least fifteen, whose blue clique contains a wheel. Thus $G$ is
nonbipartite. Every blue neighborhood has no red $C_{15}$ and no blue $C_{10}$.
Since $R(C_{15},C_{10})=19$, it has at most eighteen vertices,
and therefore $\delta(G)\ge10$.
If $G$ were disconnected, every component would have at least eleven
vertices. Three components give a blue wheel by using a hub in one
and alternating five vertices of each of two others. Thus there
would be two components. The larger has order $a\ge15$ and,
by Lemma~\ref{lem:blue-sides}, minimum red degree at least $a-5>a/2$.
Bondy's theorem supplies a red $C_{15}$, a contradiction.
If $u$ were a cut vertex, each component of $G-u$ would have at
least ten vertices. Again there are exactly two components. If the
larger has at least fifteen vertices, the same argument applies.
Otherwise both have fourteen vertices and minimum internal degree
at least nine. Either is Hamilton-connected, contrary to
Lemma~\ref{lem:core}. Hence $G$ is $2$-connected.

If $\delta(G)\ge11$, the weak pancyclicity theorem of
Brandt et al.~\cite{bfg} applies, since
$11\ge(29+2)/3$ and $G$ is nonbipartite. Dirac's circumference
bound is at least 22. The same weak pancyclicity theorem gives
girth three or four, so $G$ contains $C_{15}$.
Therefore $\delta(G)=10$.
\end{proof}

\begin{lemma}\label{prop:local15}
For a counterexample at $n=15$ and every vertex $v$ of red degree ten, its blue neighborhood has
order eighteen and contains a red $C_{14}$.
\end{lemma}
\begin{proof}
There are $29-1-10=18$ vertices in the blue neighborhood. Its blue
graph has no $C_{10}$, because $v$ would be a blue hub.
The equality $R(C_{14},C_{10})=18$ proves the assertion.
\end{proof}

\begin{lemma}
\label{lem:independent-five-separated}
Suppose disjoint vertex sets $A,B$ have all cross-edges blue,
$|B|\ge6$, and $A$ contains five mutually blue-adjacent vertices.
If there is no blue $W_{11}=K_1+C_{10}$, the blue graph on $B$
has at most one edge.
\end{lemma}
\begin{proof}
Fix a blue hub among those five vertices of $A$. If $B$ had
two distinct blue edges, their union would contain either a
blue $P_3$ or two vertex-disjoint blue $P_2$'s. Add singleton
paths to obtain four vertex-disjoint blue paths covering six
vertices of $B$. The other four vertices of the blue clique
in $A$ join these four paths cyclically into a blue $C_{10}$.
Every vertex of this cycle is blue-adjacent to the chosen hub,
a contradiction.
\end{proof}

\begin{proposition}\label{prop:local-disconnected}
Let $n\in\{14,15\}$, and let $G$ satisfy the counterexample assumptions
at $n$. Choose a vertex $v$ of minimum red degree and put
$H=G[N_{\bG}(v)]$. If the red graph $H$ is disconnected, one red
component has all but at most one vertex of $H$.
\end{proposition}
\begin{proof}
Write $r=|H|$, so $r=17$ or $18$, and let $A$ be a largest red
component. All edges between distinct red components are blue.
If $|A|\leq r-5$ and $|A|\geq5$, the sets $A$ and
$V(H)\setminus A$ each have at least five vertices, so their blue
cross-edges contain a $K_{5,5}$ and hence a blue $C_{10}$.
If instead every red component has at most four vertices, collect
whole components until their union $U$ first has at least five
vertices. Then $5\leq|U|\leq8$, while
$|V(H)\setminus U|\geq r-8\geq9$. Again the blue cross-edges
contain a $C_{10}$. Therefore, with $B=V(H)\setminus A$, we have
$|B|\leq4$.

Suppose $|B|=4$. Every edge from $A$ to $B$ is blue. If the blue
graph within $A$ contained two distinct edges, their union would
either be a blue path on three vertices or two disjoint blue edges.
In the first case, choose three further vertices of $A$; in the
second, choose two further vertices. Arrange these six vertices
in four blue paths, using the given two blue edges, and join the
four paths cyclically through the four vertices of $B$. This is a
blue $C_{10}$. Thus the red graph on $A$ is complete apart from
at most one edge. Since $|A|=r-4=n-1\geq13$, its minimum red
degree is at least $|A|-2\geq(|A|+1)/2$; by the standard
Hamilton-connectedness degree criterion, it is Hamilton-connected.
The core lemma, Lemma~\ref{lem:core}, then contradicts the
counterexample assumption.

Now suppose $|B|=3$ and let $F=\bG[A]$. If $F$ contained two
vertex-disjoint paths on three vertices, those paths, one further
vertex of $A$, and the three vertices of $B$ could be joined into
a blue $C_{10}$. Thus $F$ has no two vertex-disjoint $P_3$'s.
Two vertices of $F$ of degree at least five would supply such
paths: choose two neighbors of the first, and then two neighbors
of the second avoiding the first path. Even when the two centers
are adjacent, each has at least four other neighbors, so this
choice is possible. Hence at most one vertex of $F$ has degree
at least five. Delete that vertex if it exists, or an arbitrary
vertex otherwise, to obtain $S\subseteq A$ with
$|S|=r-4=n-1$ and $\Delta(F[S])\leq4$. The red graph $H[S]$
has minimum degree at least $|S|-5\geq(|S|+1)/2$, so it is
Hamilton-connected. Lemma~\ref{lem:core} again gives a contradiction.

Suppose finally that $|B|=2$. Then $|A|=15$ for $n=14$ and
$|A|=16$ for $n=15$. We find three vertices $x,y,z\in A$
with many blue neighbors. For $n=14$, Bondy's pancyclicity theorem
implies that $\delta(H[A])\leq7$: the only possible exception at
minimum degree at least eight is a balanced complete bipartite
graph, which has even order, whereas $|A|=15$. Choose $x$ of red
degree at most seven, so $d_{\overline{H[A]}}(x)\geq7$. The red graph on
$A-x$ has 14 vertices and no $C_{14}$; by Dirac's theorem it has
a vertex $y$ of red degree at most six. Thus $y$ has at least
seven blue neighbors in $A-x$. If the red graph on $A-\{x,y\}$
had minimum degree at least seven, it would be Hamilton-connected
on 13 vertices and Lemma~\ref{lem:core} would finish the proof.
Otherwise it has a vertex $z$ with at least six blue neighbors
in $A-\{x,y\}$.

For $n=15$, Bondy's theorem gives $\delta(H[A])\leq7$ as well:
minimum degree at least eight would give a red $C_{15}$ unless
$H[A]=K_{8,8}$. In that exception, either red bipartition class
is a blue $K_8$, and those eight vertices together with the two
vertices of $B$ form a blue $C_{10}$. Hence choose $x$ with at
least eight blue neighbors in $A$. Dirac's theorem on the
15-vertex red graph $H[A-x]$ gives a vertex $y$ with at least
seven blue neighbors in $A-x$. Finally, failure of the
Hamilton-connectedness degree criterion on the 14-vertex red
graph $H[A-\{x,y\}]$ supplies a vertex $z$ with at least six
blue neighbors there; otherwise Lemma~\ref{lem:core} applies.

Put $T=A\setminus\{x,y,z\}$. In either case, each of $x,y,z$
has at least five blue neighbors in $T$, while $|T|$ is 12 or 13.
Since the sum of the three blue-neighborhood sizes exceeds $|T|$,
two of these vertices, say $x,y$ after relabeling, have a common
blue neighbor $t\in T$. Choose two blue neighbors of the third
vertex $z$ in $T-t$, and then distinct additional blue neighbors
of $x$ and $y$ avoiding those two vertices and $t$. This yields
vertex-disjoint blue paths $P_5$ and $P_3$ within $A$:
the $P_5$ goes through $x,t,y$, and the $P_3$ is centered at $z$.
The two vertices of $B$ join their four endpoints cyclically,
forming a blue $C_{10}$. This final contradiction proves
$|B|\leq1$.
\end{proof}

\subsection{Paths through a separation}
\begin{lemma}
\label{lem:four-side-routing}
Let $B,S$ be disjoint vertex sets with $|B|=4$ and $|S|=s$.
Suppose every red neighbor of a vertex of $B$ lies in
$B\cup S$, every vertex of $B$ has red degree at least
$s+2$, and every vertex of $S$ has a red neighbor in $B$.
For any distinct $x,y\in S$ and each $q\in\{2,3,4\}$,
there is a red $x$--$y$ path with exactly $q$ internal
vertices, all in $B$.
\end{lemma}
\begin{proof}
The red graph on $B$ has minimum degree at least two,
so it is $C_4$, $K_4$ minus an edge, or $K_4$. A vertex
of red degree two in $B$ is red-complete to $S$, while
one of degree three has at most one blue neighbor in $S$.
For $C_4$, every vertex of $B$ is red-complete to $S$,
and red paths of all three orders are available.
For $K_4-uv$, the nonadjacent vertices $u,v$ are
red-complete to $S$ and are joined within $B$ by paths
on three and four vertices. To obtain a two-vertex path,
use $u$ and a vertex $r\notin\{u,v\}$: at least one of
$x,y$ is red-adjacent to $r$, since $r$ has at most one
blue neighbor in $S$. Orient the path accordingly.
For $K_4$, the red neighborhoods of $x,y$ in $B$ are
nonempty and cover $B$, so distinct endpoints can be
chosen from them and joined by paths on two, three,
or four vertices.
\end{proof}

\begin{lemma}
\label{lem:dense-eight}
Let $S,U$ be disjoint vertex sets with $|U|=8$ and
$|S|=s\in\{7,8\}$. Suppose every vertex of $U$ has at
least $s-1$ red neighbors in $S$. If $s=7$ and $G[U]$
contains a red $P_3$, then $G[S\cup U]$ contains a red
$C_{14}$. If $s=8$ and $G[U]$ contains a red edge, then
$G[S\cup U]$ contains a red $C_{15}$.
\end{lemma}
\begin{proof}
Use the given red path as one block and each unused vertex
of $U$ as a singleton block. In either case there are
$m=s-1$ blocks, of which $s-2$ are singletons. Order and
orient the blocks cyclically. To join consecutive blocks,
a gap can use any vertex of $S$ red-adjacent to its two
endpoints. Each gap has at least $s-2=m-1$ candidates.

If a singleton is red-complete to $S$, or two singletons
have the same unique blue neighbor in $S$, place a suitable
pair of singleton blocks consecutively. That gap then has
at least $s-1=m$ candidates. Hall's condition holds: any
set of at most $m-1$ gaps has at least $m-1$ candidates in
its union, and all $m$ gaps have at least $m$.

Otherwise every singleton has a unique blue neighbor in
$S$, and these neighbors are all distinct. Use any cyclic
order of the blocks. Hall's condition could fail only for
the full set of $m$ gaps, and then every gap must have the
same $m-1$ candidates, omitting the same two vertices of
$S$. Each singleton would consequently have its blue
neighbor among those two omitted vertices, impossible
because there are $s-2\ge5$ singletons with distinct blue
neighbors. Thus Hall's condition again holds.

Choose distinct vertices of $S$ for the $m$ gaps. Joining
the blocks through these vertices gives a red cycle using
all eight vertices of $U$ and $s-1$ vertices of $S$, of
order $s+7$, as required.
\end{proof}

\begin{lemma}
\label{lem:two-dirac-components}
Let $h\ge3$. Suppose disjoint sets $A,B$ each have $2h$
vertices, with $\delta(G[A]),\delta(G[B])\ge h$. Suppose
there are distinct vertices $s_1,\ldots,s_{h+1}$ outside
$A\cup B$, distinct $a_1,\ldots,a_{h+1}\in A$, and
distinct $b_1,\ldots,b_{h+1}\in B$, such that every
$a_i s_i$ and $s_i b_i$ is an edge. Then $G$ contains
$C_{2h+4}$.
\end{lemma}
\begin{proof}
Whenever $G[A]$ has a Hamilton path from $a_i$ to $a_j$
and $b_i b_j$ is an edge, the path, that edge, and the
four attachment edges form $C_{2h+4}$. The same statement
holds with $A,B$ interchanged. Suppose no such cycle exists.

The minimum-degree bounds imply that both induced graphs
satisfy Ore's degree-sum condition. We use Theorem~3 of
Shih et al.~\cite{SSK}: a graph of order $q$
satisfying $d(x)+d(y)\ge q$ for every nonadjacent pair
is Hamilton-connected unless it has one of the forms
\[
 H_2\vee(K_p\cup K_{q-2-p})\quad\hbox{or}\quad
 H_{q/2}\vee \overline{K}_{q/2},
\]
where $H_j$ is an arbitrary graph on $j$ vertices and
$\vee$ denotes the join. In the first family the two
cliques are nonempty, and the second family has even order.
At order $2h$ with minimum degree at least $h$, the first
form reduces to $H_2\vee(2K_{h-1})$.

If $G[A]$ is Hamilton-connected, all pairs among
$b_1,\ldots,b_{h+1}$ must be nonedges. This is impossible:
an independent set in a $2h$-vertex graph of minimum degree
$h$ has at most $h$ vertices. Thus neither $G[A]$ nor $G[B]$ is Hamilton-connected.

Suppose $G[A]=H_2\vee(2K_{h-1})$. Since $h\ge3$, any
vertex in either clique is the initial vertex of a Hamilton
path ending at any other vertex. Indeed, the two vertices
of $H_2$ can join the two clique paths; if the endpoints
lie in different cliques, split the terminal clique into
two nonempty paths. At least one marked vertex $a_i$ lies
outside $H_2$. Hence $b_i$ is nonadjacent to each of the
other $h$ marked vertices of $B$. But every vertex of
$G[B]$ has at most $h-1$ nonneighbors, a contradiction.
The first exceptional form is therefore impossible on
either side.

Consequently $A$ has a partition $X\cup I$ and $B$ a
partition $Y\cup J$, with all four parts of order $h$,
such that $I,J$ are independent and all $X$--$I$ and
$Y$--$J$ edges are present. The marked vertices in $A$
meet both $X$ and $I$. A Hamilton path of the spanning
$K_{h,h}$ joins every marked vertex in $X$ to every
marked vertex in $I$. Let $B_X,B_I$ be the corresponding
sets of marked vertices in $B$. Thus all $B_X$--$B_I$
pairs are nonedges. Their union has $h+1$ vertices and
therefore meets $J$. If, say, $B_X$ meets $J$, the complete
$Y$--$J$ adjacency forces $B_I\subseteq J$, and then
forces $B_X\subseteq J$, since $B_I$ is nonempty. This
would put $h+1$ distinct vertices in $J$, a contradiction.
The case where $B_I$ meets $J$ is symmetric.
\end{proof}

\section{Connectivity for the eleven-vertex wheel}
\label{sec:w11-connectivity}
We prove that a counterexample satisfies $\kappa(G)\ge8$ when
$n=14$ and $\kappa(G)\ge9$ when $n=15$. These bounds will be
combined with K\"onig's theorem in the next two sections.

The argument proceeds by excluding vertex cuts in increasing order of size. The two
cases have different degree bounds: $|G|=27$ and $\delta(G)=9$
when $n=14$, whereas $|G|=29$ and $\delta(G)=10$ when $n=15$.
For a cut $S$ of order $s$, each component $D$ of $G-S$ satisfies
$|D|\ge\delta(G)-s+1$. If at least five vertices lie in the other
components, the blue-wheel observation also gives
$\delta(G[D])\ge|D|-5$. These bounds restrict the possible component orders;
small components are handled by the path constructions from
Section~\ref{sec:w11-setup}.
Table~\ref{tab:w11-cut-stages} records the resulting cases and
where each is excluded. The proofs are given in the cited propositions.

\begin{table}[tbp]
\centering
\small
\caption{Successive cut exclusions in Section~\ref{sec:w11-connectivity}.
The entries are component orders after deleting a cut $S$ of order
$s$, following the size reductions in the cited propositions.}
\label{tab:w11-cut-stages}
\begin{tabular}{@{}c c p{0.43\textwidth} p{0.39\textwidth}@{}}
\toprule
$n$ & $s$ & Possible component orders & Exclusion \\
\midrule
14 & $\le5$ &
$s\le2$: none;
$s=3$: $(12,12)$;
$s=4$: $(12,11)$;
$s=5$: $(12,10),(11,11)$.
& Proposition~\ref{prop:global-six-connected}. \\
\addlinespace
14 & $6$ &
$(17,4),(12,9)$
(Proposition~\ref{prop:six-cut14}).
& $(12,9)$: Proposition~\ref{prop:six-cut14-remaining};
$(17,4)$: Proposition~\ref{prop:global-seven-connected14}. \\
\addlinespace
14 & $7$ &
$(17,3),(16,4),(11,9),(10,10)$
(Proposition~\ref{prop:next-cut-sizes14}).
& All but $(10,10)$:
Proposition~\ref{prop:seven-cut14-balanced};
$(10,10)$: Proposition~\ref{prop:global-eight-connected14}. \\
\midrule
15 & $\le6$ &
$s\le2$: none;
$s=3$: $(13,13)$;
$s=4$: $(13,12)$;
$s=5$: $(13,11),(12,12)$;
$s=6$: $(13,10),(12,11)$.
& Proposition~\ref{prop:global-six-connected15}. \\
\addlinespace
15 & $7$ &
$(18,4),(13,9),(11,11)$
(Proposition~\ref{prop:seven-cut15}).
& $(18,4)$: Proposition~\ref{prop:seven-cut18};
$(13,9),(11,11)$:
Proposition~\ref{prop:global-eight-connected15}. \\
\addlinespace
15 & $8$ &
$(18,3),(17,4)$
(Proposition~\ref{prop:next-cut-sizes}).
& $(17,4)$: Proposition~\ref{prop:eight-cut17};
$(18,3)$: Proposition~\ref{prop:global-nine-connected15}. \\
\bottomrule
\end{tabular}
\end{table}

\subsection{Eight-connectivity at \texorpdfstring{$n=14$}{n=14}}
Throughout this subsection, $G$ is a counterexample at $n=14$.
Thus $|G|=27$ and $\delta(G)=9$ by Corollary~\ref{cor:small14}.

\begin{proposition}
\label{prop:global-six-connected}
Under the counterexample assumptions for $n=14$, the red graph
$G$ is at least $6$-connected.
\end{proposition}
\begin{proof}
Suppose that a set $S$ of $s\leq5$ vertices disconnects $G$.
Every component $A$ of $G-S$ has order at least $10-s\geq5$,
because $\delta(G)=9$ and its vertices have no red neighbors
outside $A\cup S$. Three components would give a blue
$W_{11}$: take a hub in one and alternate five vertices from
each of two others around the rim. Thus $G-S$ has exactly two
components $A,B$.

If some vertex of $A$ had five blue neighbors in $A$, those
five and five vertices of $B$ would form a blue $C_{10}$ with
that vertex as hub. Hence $\delta(G[A])\geq |A|-5$; the same
holds for $B$. A component of order $a\geq14$ would have
$\delta\geq a-5\geq a/2$. Bondy's pancyclicity theorem then
gives a red $C_{14}$, including in its exceptional complete
bipartite case. A component of order $13$ has minimum degree
at least eight and is Hamilton-connected, contrary to
Lemma~\ref{lem:core}. Therefore both components have order
at most $12$. Since $|A|+|B|=27-s$, we have $s\geq3$.

For any $z\in S$, the vertex $z$ cannot have five blue
neighbors in each of $A$ and $B$, since they would form a
blue rim with hub $z$. Assign $z$ to a component where it has
at most four blue neighbors; if assigned to a component of
order $a$, it has at least $a-4$ red neighbors there.

If $s=3$, the component orders are $12,12$. Two vertices of
$S$ are assigned to one component, say $A$. In the graph
induced by $A$ and these two vertices, every vertex of $A$
has red degree at least $12-5=7$, while the two added
vertices have at least eight red neighbors in $A$.
Dirac's theorem supplies a red $C_{14}$.

If $s=4$, the component orders are $12,11$. If two vertices
of $S$ are assigned to the $12$-vertex component, the preceding
Dirac argument applies. Otherwise at least three vertices
are assigned to the $11$-vertex component $B$. In the graph
induced by $B$ and these three vertices, each added vertex
has at least seven red neighbors in $B$. Each vertex of $B$
has red degree at least eight in this induced graph, since
at most one vertex of $S$ was omitted. Again Dirac's theorem gives
a red $C_{14}$.

If $s=5$ and the component orders are $11,11$, three vertices
of $S$ are assigned to one component. These three have at
least seven red neighbors in it, and its vertices retain red
degree at least $9-2=7$ after the other two vertices of $S$
are omitted. Dirac's theorem gives a red $C_{14}$. The only other
possibility is that the component orders are $12,10$. Two
vertices assigned to the $12$-vertex component give the same
contradiction by Dirac's theorem, so at least four vertices of $S$ are
assigned to the $10$-vertex component $B$. In the graph
induced by $B$ and these four vertices, the four added
vertices have red degree at least six, while every vertex
of $B$ has red degree at least eight, because at most one
vertex of $S$ was omitted. If its ordered degree sequence is
$d_1\leq\cdots\leq d_{14}$, then $d_i>i$ for every
$i<7$: the first four degrees are at least six and the next
two are at least eight. Chv\'atal's degree-sequence
theorem~\cite{Chvatal}
therefore makes this graph Hamiltonian, again producing a
red $C_{14}$. All cases are impossible.
\end{proof}

\begin{proposition}
\label{prop:six-cut14}
Under the counterexample assumptions for $n=14$, if a set $S$
of six vertices disconnects $G$, then $G-S$ has exactly two
components, of orders $(17,4)$ or $(12,9)$.
\end{proposition}
\begin{proof}
Every component of $G-S$ has order at least $9-6+1=4$.
If there are at least three components, either all of them
have at least five vertices, giving a blue wheel directly, or
some component $D$ has order four. Fix $v\in D$. The other
components, whose orders sum to $17$ and are each at least
four, can be partitioned into two unions of at least five
vertices unless they are exactly two components of orders
$(13,4)$. Such a partition gives a blue $K_{5,5}$ in the
blue neighborhood of $v$, hence a blue wheel. In the
exceptional $(13,4)$ case, the four-vertex component is
blue-complete to the $13$-vertex component $A$. If $A$
contained two distinct blue edges, those edges, together
with four vertices of the smaller component, would give a
blue $C_{10}$ in the blue neighborhood of $v$: arrange six
vertices of $A$ into four blue paths and join them cyclically
through the four outside vertices. Hence $G[A]$ is a red
$K_{13}$ with at most one edge deleted. It is
Hamilton-connected, contradicting Lemma~\ref{lem:core}.
Thus $G-S$ has two components.

Their orders sum to $27-6=21$. If the smaller has order at least
five, the blue-wheel observation gives internal red minimum
degree at least $a-5$ in a component of order $a$. A component
of order at least $14$ then contains a red $C_{14}$ by Bondy's
theorem; its balanced complete-bipartite exception cannot attain
minimum degree $a-5>a/2$. A component of order $13$ is
Hamilton-connected and contradicts Lemma~\ref{lem:core}.
Therefore the only possible component orders with the smaller
at least five are $(12,9)$ and $(11,10)$.

To exclude $(11,10)$, assign each vertex of $S$ to a component
where it has at most four blue neighbors. If at least four are
assigned to the $10$-vertex component, select four. These and
that component induce a graph of order $14$ in which the four
selected vertices have red degree at least six and the ten
component vertices retain red degree at least $9-2=7$.
Hence $d_i>i$ for $1\le i\le6$, and Chv\'atal's theorem gives
a red $C_{14}$. Otherwise at least three vertices of $S$ are
assigned to the $11$-vertex component. Select three. Each has
at least seven red neighbors there. Component vertices have
internal red degree at least six, and at least seven of them
have an additional neighbor among the selected vertices of $S$. So at most
four vertices have degree six and all others have degree at
least seven. Again $d_i>i$ for $1\le i\le6$, giving a red
$C_{14}$. If the smaller component has order four, the other
has order $17$, completing the classification.
\end{proof}

\begin{proposition}
\label{prop:six-cut14-remaining}
Under the counterexample assumptions for $n=14$, a six-vertex
cut cannot leave components of orders $(12,9)$. Consequently,
the only possible six-cut component orders are $(17,4)$.
\end{proposition}
\begin{proof}
Suppose a six-vertex cut $S$ leaves components $A,B$ with
$|A|=12$ and $|B|=9$. The blue-wheel observation gives
$\delta(G[A])\ge12-5=7$ and $\Delta(\bG[B])\le4$.
Williamson's panconnectedness theorem~\cite{Williamson}
applies to $G[A]$, because $7=(|A|+2)/2$. In particular,
any two distinct vertices of $A$ can be joined by a red
path on ten vertices.

Since $G$ is $6$-connected by
Proposition~\ref{prop:global-six-connected}, both $G[A,S]$ and
$G[B,S]$ have matchings of size six. Indeed, if either bipartite
graph had matching number at most five, K\"onig's theorem would
give a vertex cover of size at most five. Deleting this cover
would separate nonempty subsets of $A$ and $B$, a contradiction.
Choose one such matching in each graph; both saturate $S$. Denote their
matched endpoints for $s\in S$ by $a_s\in A$ and $b_s\in B$.

The six distinct vertices $b_s$ cannot be independent in $G$,
because every vertex of $B$ has at most four blue neighbors
there. Choose $s\ne t$ with $b_sb_t$ red. A red path in $A$
on ten vertices from $a_s$ to $a_t$, followed by the red
edges $a_ss$, $sb_s$, $b_sb_t$, $b_tt$, and $ta_t$, gives a
red $C_{14}$, contrary to assumption.
\end{proof}

\begin{proposition}
\label{prop:six-cut17-cycle}
Assume the $n=14$ counterexample conditions and suppose a
six-vertex cut $S$ leaves red components $A,B$ of orders
$17,4$. Fix a red matching of size six between $A$ and $S$
and any red $C_{12}$ in $A$. Exactly two matching endpoints
lie on the cycle; they are adjacent or antipodal on it.
\end{proposition}
\begin{proof}
Every vertex of $B$ has red degree nine and at most
$3+6=9$ possible red neighbors, so $B$ is a red $K_4$
complete to $S$. The $6$-connectivity of $G$ and
K\"onig's theorem imply that $G[A,S]$ has a matching of size six:
if $\nu(G[A,S])\le5$, a vertex cover of size at most five would
separate a nonempty subset of $A$ from $B$.

For any two red $A$--$S$ edges with distinct endpoints,
a red path on $k\in\{8,9,10,11\}$ vertices between their
$A$-endpoints would form a red $C_{14}$ when joined to a red
path between the corresponding vertices of $S$ with
$q=12-k\in\{4,3,2,1\}$ internal vertices in $B$. Every vertex of $B$ has blue neighborhood
exactly $A$, so the blue graph on $A$ has no $C_{10}$.
The equality $R(C_{12},C_{10})=16$~\cite{FSR} supplies
a red $C_{12}$ in $A$. Five vertices of $A$ lie outside it.

Among the six matching endpoints at least one lies on the
cycle. Two cycle endpoints at cyclic distance $2,3,4$, or
$5$ would be joined by a cycle arc on $11,10,9$, or $8$
vertices, respectively, which is forbidden. Thus two cycle
endpoints can only be adjacent or antipodal. An antipodal
pair permits no third endpoint at an allowed distance from
both, and three vertices cannot be pairwise adjacent on a
$C_{12}$. Hence there cannot be three cycle endpoints.

Suppose just one matching endpoint, say $c_0$, lies on the
cycle $C=c_0c_1\cdots c_{11}c_0$. Then all five vertices
$T=A\setminus V(C)$ are matched to distinct vertices of $S$.
No $t\in T$ can have a red edge to any of
$F=\{c_3,\ldots,c_9\}$: such an edge, followed by one of
the two arcs to $c_0$, would yield a forbidden path on
$8,9,10$, or $11$ vertices. Hence every $T$--$F$ edge is
blue. The blue $K_{5,7}$ between $T$ and $F$ contains a
blue $C_{10}$, with any vertex of $B$ as blue hub. This
contradiction shows there are two cycle endpoints.

\end{proof}

\begin{proposition}
\label{prop:global-seven-connected14}
Under the counterexample assumptions for $n=14$, the red graph
$G$ is at least $7$-connected.
\end{proposition}
\begin{proof}
Proposition~\ref{prop:global-six-connected} gives
$\kappa(G)\ge6$. Suppose a six-vertex cut $S$ exists.
Propositions~\ref{prop:six-cut14} and
\ref{prop:six-cut14-remaining} leave only components $A,B$
of orders $17,4$. By Proposition~\ref{prop:six-cut17-cycle},
there is a red $C_{12}=c_0c_1\cdots c_{11}c_0$ in $A$
and a size-six red $A$--$S$ matching whose two endpoints
on the cycle are adjacent or antipodal. In the adjacent
case relabel them $c_0,c_1$ and put
$F=\{c_3,c_4,c_5,c_8,c_9,c_{10}\}$. In the antipodal
case relabel them $c_0,c_6$ and put
$F=\{c_2,c_3,c_4,c_8,c_9,c_{10}\}$.
The other four
matching endpoints form a set $T\subseteq A\setminus V(C)$.

As in the proof of Proposition~\ref{prop:six-cut17-cycle},
no two red $A$--$S$ edges with distinct endpoints can have
their $A$-endpoints joined by a red path on $8,9,10$, or
$11$ vertices. Hence every edge from $T$ to $F$ is blue.
Indeed, each $f\in F$ is at cyclic distance $3,4$, or $5$
from at least one of the two matched cycle endpoints;
attaching a matched outside
vertex to $f$ would yield one of the forbidden paths.

Every $S$--$F$ edge is blue as well. Each $f\in F$ is at
cyclic distance $2,3,4$, or $5$ from both matched cycle
endpoints. If $sf$ were red, choose one of these endpoints whose matching
partner is different from $s$. The longer arc between $f$
and that cycle vertex is a forbidden path on $8,9,10$,
or $11$ vertices.

If some $t\in T$ had a blue neighbor $s\in S$, then the
four vertices of $B$ together with $s$ and any five
vertices of $F$ would form a blue $K_{5,5}$ inside the
blue neighborhood of $t$: all $B$--$F$, $S$--$F$,
$T$--$F$, and $T$--$B$ edges involved are blue.
This gives a blue $W_{11}$. Thus every $T$--$S$ edge is
red. Every $B$--$S$ edge is also red, as established in
Proposition~\ref{prop:six-cut17-cycle}. Therefore the red
graph on $S\cup T\cup B$ contains a $K_{6,8}$ with sides
$S$ and $T\cup B$. Three vertices of the red clique $B$
form a red $P_3$ inside the eight-vertex side. Use its
two edges and alternate the six vertices of $S$ with the
remaining five vertices of $T\cup B$ to form a red
$C_{14}$. This contradiction excludes every six-vertex
cut.
\end{proof}

\begin{proposition}
\label{prop:next-cut-frontier14}
A seven-vertex cut either leaves a component of order at most four,
or leaves exactly two components of orders $(11,9)$ or $(10,10)$.
\end{proposition}
\begin{proof}
Let $|S|=7$. If every component of
$G-S$ has at least five vertices, three components give a
blue wheel. With two components of orders $a\ge b\ge5$,
we have $a+b=20$ and internal red minimum degree at least
$a-5$ in the larger component. If $a\ge14$, Bondy's theorem
gives a red $C_{14}$; if $a=13$, the component is
Hamilton-connected on $n-1$ vertices, contradicting
Lemma~\ref{lem:core}. Thus the candidate pairs are
$(12,8)$, $(11,9)$, and $(10,10)$.

For $(12,8)$, $7$-connectivity and K\"onig's theorem
give red matchings of size seven from $S$ into each
component, both saturating $S$. The $12$-vertex component
has red minimum degree at least seven, so Williamson's
theorem~\cite{Williamson} makes it panconnected. The
seven matched vertices of the $8$-vertex component cannot
be independent in $G$ because every component vertex has at
most four blue neighbors inside it. Choose two matched
vertices joined by a red edge. A red path on ten vertices
between their corresponding matched vertices in the larger
component, followed by two vertices of $S$ and the red edge,
forms a red $C_{14}$. This excludes $(12,8)$.

\end{proof}

\begin{proposition}
\label{prop:next-cut-sizes14}
Every seven-vertex cut leaves exactly two components, with orders
$(17,3)$, $(16,4)$, $(11,9)$, or $(10,10)$.
\end{proposition}
\begin{proof}
Let $S$ be a cut of order seven. Since $\delta(G)=9$, every
component of $G-S$ has at least three vertices. Suppose there
are at least three components. If a component $D$ has order
three, every $v\in D$ has exactly nine red neighbors,
namely $(D-v)\cup S$. Its blue neighborhood is the union
of the other components, each of order at least three. This
contradicts Proposition~\ref{prop:local-disconnected}.

Thus every component has at least four vertices. There cannot
be two components of order at least five, since a hub in a
third and five vertices from each of those two would form a
blue wheel. Hence two components have order four. Let $U$
be their union. Each vertex of $U$ has at least six red
neighbors in $S$ and internal red degree at least two in
its own component. In particular, $G[U]$ contains a red $P_3$.
Lemma~\ref{lem:dense-eight}, with $s=7$, gives a red $C_{14}$.
This contradiction proves that $G-S$ has two components.
Their orders sum to twenty. The claimed list follows from
Proposition~\ref{prop:next-cut-frontier14} and the lower
bound of three on each component order.
\end{proof}

\begin{proposition}
\label{prop:seven-cut14-balanced}
Under the $n=14$ counterexample assumptions, every
seven-vertex cut leaves exactly two components of order ten.
\end{proposition}
\begin{proof}
By Proposition~\ref{prop:next-cut-sizes14}, the candidate
component orders are $(17,3)$, $(16,4)$, $(11,9)$, and
$(10,10)$. Write the components $A,B$, with $A$ the larger.
The red $A$--$S$ bipartite graph always has a matching of
size seven, by $7$-connectivity and K\"onig's theorem.

Suppose first that $(|A|,|B|)=(11,9)$. The blue-wheel
observation gives $\delta(G[A])\ge6$, so $G[A]$ is
Hamilton-connected. If some $b\in B$ had two distinct
red neighbors $s,t\in S$, use the distinct $A$-endpoints
of their matching edges and a spanning red path in $A$.
Together with $s,b,t$, this gives a red $C_{14}$.
Thus each vertex of $B$ has at most one red neighbor in
$S$. Its red degree is at least nine, so it must have
exactly one such neighbor and all eight possible red
neighbors in $B$. Hence $G[B]=K_9$.
There is also a size-seven red $B$--$S$ matching. Choose
two distinct vertices $s,t\in S$ and their matched endpoints
$a_s,a_t\in A$ and $b_s,b_t\in B$. The vertices $a_s,a_t$
have a common red neighbor in $A$, since $|A|=11$ and
$\delta(G[A])\ge6$. A red path on three vertices between
$a_s,a_t$, a spanning red path in $B$ between $b_s,b_t$,
and the four matching edges give a red cycle on
$3+9+2=14$ vertices, a contradiction.

Next suppose $(|A|,|B|)=(16,4)$. Every vertex of $S$ has
a red neighbor in $B$, since deleting the other six
would otherwise disconnect $G$. Apply
Lemma~\ref{lem:four-side-routing}, with $|S|=7$ and
$\delta(G)=9$. It follows that no two red $A$--$S$
edges with distinct endpoints can have their $A$-endpoints
joined by a red path on $8,9$, or $10$ vertices.
The blue graph on $A$ has no $C_{10}$, so
$R(C_{12},C_{10})=16$ gives a red
$C=c_0c_1\cdots c_{11}c_0$ in $A$.
At least three matching endpoints lie on $C$. Their
pairwise cyclic distances cannot be $3,4$, or $5$,
because the longer arcs have $10,9$, or $8$ vertices.
The permitted distances are $1,2$, and $6$. An antipodal
pair permits no third endpoint; otherwise there are at
most three endpoints, and three are consecutive. Hence
exactly three lie on $C$, say $c_0,c_1,c_2$, and all four
vertices of $T=A\setminus V(C)$ are matched.

The vertex $c_4$ is blue-adjacent to all of $S$ and $T$.
A red edge to $S$ would combine with the longer arc to
$c_0$ or $c_1$, choosing a different matching partner,
to violate the forbidden path orders. A red edge from
$t\in T$ to $c_4$ would extend the longer $c_4$--$c_0$
arc to a forbidden path on ten vertices. Edges from
$c_4$ to $B$ are blue as well. The four chords
$c_4c_1,c_4c_7,c_4c_8,c_4c_9$ must also be blue:
otherwise they shortcut the spanning $c_1$--$c_0$ arc
to paths on $10,10,9,8$ vertices, respectively. Thus
$d_G(c_4)\le11-4=7<9$, a contradiction.

Finally suppose $(|A|,|B|)=(17,3)$. Here $B$ is a red
$K_3$ complete to $S$, since a vertex of $B$ has exactly
$2+7=9$ possible red neighbors. Using paths with one,
two, or three internal vertices in $B$ forbids red paths on
$11,10$, or $9$ vertices between distinct red attachment
endpoints in $A$. Again $A$ has a red $C_{12}$, and at
least two matching endpoints lie on it. If exactly two
lie on the cycle, all five outside vertices are matched.
Choose one cycle endpoint as $c_0$. Each matched outside
vertex is blue-adjacent to
$\{c_3,c_4,c_5,c_7,c_8,c_9\}$, since a red edge there
would give a forbidden path to $c_0$. This blue
$K_{5,6}$ supplies a wheel with a hub in $B$.

There are therefore at least three cycle endpoints.
Permitted cyclic distances are $1,5$, and $6$.
Three such endpoints must include an adjacent pair: without
an adjacent pair, after fixing one endpoint as $c_0$, the
other two would lie in $\{c_5,c_6,c_7\}$, which is impossible
because their distance would be one or two. Relabel an
adjacent pair as $c_0,c_1$. An additional endpoint must then
be $c_6$ or $c_7$, so there are three or four endpoints.
Let $T$ be the matched vertices outside the cycle;
at most two other outside vertices are unmatched.
The same path restriction makes every $c_4$--$S$ and
$c_4$--$T$ edge blue, and $c_4$ has no red neighbor
in $B$. Moreover the five chords
$c_4c_2,c_4c_6,c_4c_1,c_4c_7,c_4c_8$ are blue,
since they shortcut the spanning $c_1$--$c_0$ arc to
paths on $11,11,10,10,9$ vertices, respectively.
Consequently $d_G(c_4)\le11-5+2=8<9$, the final
contradiction. Only the pair $(10,10)$ remains.
\end{proof}

\begin{proposition}
\label{prop:global-eight-connected14}
Under the counterexample assumptions for $n=14$, the red
graph $G$ is at least $8$-connected.
\end{proposition}
\begin{proof}
We already have $\kappa(G)\ge7$. If a seven-vertex cut
$S$ existed, Proposition~\ref{prop:seven-cut14-balanced}
would give exactly two components $A,B$ of order ten.
All cross-edges are blue, so the blue-wheel observation
gives $\delta(G[A]),\delta(G[B])\ge5$. By
$7$-connectivity and K\"onig's theorem, red matchings
from $S$ to $A$ and from $S$ to $B$ both saturate $S$.
Choose any six vertices of $S$ and their matched endpoints.
Lemma~\ref{lem:two-dirac-components}, with $h=5$, gives
a red $C_{14}$, a contradiction. Thus no seven-vertex
cut exists.
\end{proof}

\subsection{Nine-connectivity at \texorpdfstring{$n=15$}{n=15}}
Throughout this subsection, $G$ is a counterexample at $n=15$.
Thus $|G|=29$ and $\delta(G)=10$ by Corollary~\ref{cor:small}.

\subsubsection{Excluding cuts of order at most seven}
\begin{proposition}
\label{prop:global-six-connected15}
Under the counterexample assumptions for $n=15$, the red graph
$G$ is at least $7$-connected.
\end{proposition}
\begin{proof}
Suppose that $S$ is a vertex cut of order $s\leq5$.
Every component of $G-S$ has order at least $11-s\geq6$,
because $\delta(G)=10$. Three components would give a blue
$W_{11}$ by using a hub in one and five rim vertices in each
of two others. Thus there are exactly two components $A,B$.
By Lemma~\ref{lem:blue-sides}, every component of order $a$
has minimum red degree at least $a-5$.
If $a\geq15$, Bondy's theorem gives a red $C_{15}$ unless
the component is $K_{a/2,a/2}$. In the latter case, $a\geq16$
and a vertex has at least seven blue neighbors in its own
bipartition class; five of these and five vertices of the
other component make a blue $W_{11}$. If $a=14$, then
$\delta(G[A])\geq9$, so $G[A]$ is a Hamilton-connected
subgraph of order $n-1$, contrary to Lemma~\ref{lem:core}.
Consequently both components have order at most $13$,
and $s\geq3$.

Each $z\in S$ must have at most four blue neighbors in at
least one component; assign it to one such component. If a
component has order $a$, an assigned vertex has at least
$a-4$ red neighbors there.

For $s=3$, the component orders are $13,13$. At least two
vertices of $S$ are assigned to one component. Its internal
minimum red degree is at least eight, while each of the two
vertices of $S$ has at least nine red neighbors there. Dirac's
theorem gives a red $C_{15}$ on these $15$ vertices.

For $s=4$, the orders are $13,12$. Two vertices assigned to
the $13$-vertex component give the same contradiction.
Otherwise three are assigned to the $12$-vertex component.
They each have at least eight red neighbors in it, while
each of its vertices retains red degree at least $10-1=9$
after the fourth vertex of $S$ is omitted. Dirac's theorem gives $C_{15}$.

For $s=5$, the orders are either $12,12$ or $13,11$.
In the first case, three vertices of $S$ are assigned to one
component. They have at least eight red neighbors there,
and each component vertex retains degree at least $10-2=8$;
Dirac's theorem again gives $C_{15}$. In the second case, two vertices of $S$ assigned to the $13$-vertex component give a
contradiction by Dirac's theorem. Otherwise four are assigned to the
$11$-vertex component $B$. On the $15$ vertices formed by
$B$ and these four vertices, the latter have red degree at
least seven, while every vertex of $B$ has red degree at
least $10-1=9$. In the ordered degree sequence
$d_1\leq\cdots\leq d_{15}$, we have $d_i>i$ for each
$i<15/2$: the first four entries are at least seven, and
the next three are at least nine. Chv\'atal's
 degree-sequence theorem~\cite{Chvatal} therefore gives a
 Hamiltonian red cycle, necessarily a $C_{15}$.

It remains to exclude a cut $S$ of order six. Every component of
$G-S$ has at least $10-6+1=5$ vertices. Thus three components
would immediately give a blue wheel, and there are exactly two.
The observation above gives internal minimum red degree at least
$a-5$ in a component of order $a$. A component of order at least
$15$ has a red $C_{15}$ by Bondy's theorem: its balanced
complete-bipartite exception is impossible because $a-5>a/2$
for $a\ge15$. A component of order $14$ is Hamilton-connected,
since its minimum degree is at least nine, and contradicts
Lemma~\ref{lem:core}. The component orders sum to $29-6=23$,
so the remaining pairs are $(13,10)$ and $(12,11)$.

Assign each vertex of $S$ to a component in which it has at most
four blue neighbors. In the $(13,10)$ case, two vertices assigned
to the $13$-vertex component give a $15$-vertex red graph of
minimum degree at least eight, hence a $C_{15}$ by Dirac's theorem.
Otherwise at least five vertices are assigned to the $10$-vertex
component. Together they induce a graph of order $15$ in which
the five assigned vertices have red degree at least six, while
all ten component vertices have red degree at least $10-1=9$.
Consequently $d_i>i$ for $1\le i\le7$, so Chv\'atal's criterion
gives a $C_{15}$.

In the $(12,11)$ case, suppose first that at least three vertices of $S$ are assigned to the $12$-vertex component. Choose three
of them. They and that component induce a $15$-vertex graph in
which each selected vertex of $S$ has at least eight red neighbors.
Every component vertex has internal red degree at least seven;
at least eight component vertices have an additional red neighbor
among the three selected vertices of $S$. Therefore at most four
vertices have degree as low as seven, and all the others have
degree at least eight. Again $d_i>i$ for $1\le i\le7$.
If fewer than three vertices of $S$ are assigned to that component,
at least four are assigned to the $11$-vertex component. Four of
these and that component induce a $15$-vertex graph. The four
selected vertices each have at least seven red neighbors in the
component, and each component vertex retains at least
$10-2=8$ red neighbors after the other two vertices of $S$ are
omitted. Here too $d_i>i$ for $1\le i\le7$. Chv\'atal's theorem
contradicts the absence of a red $C_{15}$ in both subcases.
\end{proof}

\begin{proposition}
\label{prop:seven-cut15}
Under the counterexample assumptions for $n=15$, if a set $S$
of seven vertices disconnects $G$, then $G-S$ has exactly two
components, of orders $(18,4)$, $(13,9)$, or $(11,11)$.
\end{proposition}
\begin{proof}
Every component of $G-S$ has order at least $10-7+1=4$.
Suppose there are at least three components.
Three components of order at least five would give a blue
wheel. If instead there is a component $D$ of order four,
fix $v\in D$. The other components have total order $18$
and individual orders at least four. Unless they are exactly
two components of orders $(14,4)$, they can be partitioned
into two unions of at least five vertices. This gives a blue
$K_{5,5}$ in the blue neighborhood of $v$, hence a blue
wheel. In the exceptional $(14,4)$ case, two distinct blue
edges in the $14$-vertex component $A$ could be joined
through the four-vertex component to make a blue $C_{10}$
in the blue neighborhood of $v$. Thus $G[A]$ is a red
$K_{14}$ with at most one edge deleted. It is
Hamilton-connected on $n-1=14$ vertices, contradicting
Lemma~\ref{lem:core}. Thus there are exactly two components.

The orders sum to $29-7=22$. If the smaller has order at least
five, the blue-wheel observation gives internal red minimum
degree at least $a-5$ in a component of order $a$. A component
of order at least $15$ contains a red $C_{15}$ by Bondy's
theorem, while one of order $14$ is Hamilton-connected and
contradicts Lemma~\ref{lem:core}. The remaining size pairs
are $(13,9)$, $(12,10)$, and $(11,11)$.

Suppose the pair is $(12,10)$. Assign each of the seven vertices of $S$ to a component in which it has at most four blue
neighbors. If at least three are assigned to the $12$-vertex
component, select three. In their induced graph of order $15$,
the selected vertices have red degree at least eight. The
component vertices have internal red degree at least seven,
and at least eight of them gain a neighbor among the selected vertices of $S$.
Thus at most four vertices have degree seven; all others have
degree at least eight. If fewer than three are assigned there,
at least five are assigned to the $10$-vertex component. Choose
five. Each has at least six red neighbors in that component,
whose ten vertices retain red degree at least $10-2=8$.
In either case the ordered degree sequence satisfies $d_i>i$
for $1\le i\le7$. Chv\'atal's theorem gives a red $C_{15}$,
excluding $(12,10)$. A smaller component of order four forces
the pair $(18,4)$.
\end{proof}

\begin{proposition}
\label{prop:seven-cut18}
Under the $n=15$ counterexample assumptions, no seven-vertex
cut can leave red components of orders $(18,4)$.
\end{proposition}
\begin{proof}
Suppose to the contrary that a seven-vertex cut $S$ leaves
components $A,B$ of orders $18,4$, respectively.
Every vertex of $B$ has at most $3+7=10$ possible red
neighbors and has red degree ten, so $B$ is a red $K_4$
complete to $S$. If the red $A$--$S$ bipartite graph had
matching number at most six, K\"onig's theorem would give
a vertex cover of size at most six. Deleting that cover would
separate nonempty subsets of $A$ and $B$, contradicting
Proposition~\ref{prop:global-six-connected15}. Thus a
size-seven matching exists. By Lemma~\ref{prop:local15},
applied to any vertex of $B$, choose a red cycle
$C=c_0c_1\cdots c_{13}c_0$ in $A$. Put $T=A\setminus V(C)$;
thus $|T|=4$.

Let $as$ and $a's'$ be any red $A$--$S$ edges with
$a\ne a'$ and $s\ne s'$. A red path between $a,a'$ on $k$
vertices, with $9\le k\le12$, would make a red $C_{15}$:
join its endpoints to $s,s'$ and connect $s$ to $s'$ through
$q=13-k\in\{1,2,3,4\}$ distinct vertices of the red clique
$B$. Hence no such path exists.

Of the seven matching endpoints in $A$, at least three lie
on this cycle because only four
vertices of $A$ lie outside it. Two cycle endpoints at cyclic
distance $3,4,5$, or $6$ would have a cycle arc on
$12,11,10$, or $9$ vertices, respectively, contradicting the
previous paragraph. Thus distinct cycle endpoints can have
cyclic distance only $1,2$, or $7$. An antipodal pair at
distance seven permits no third endpoint with this property;
otherwise at most three endpoints can occur, and three must
be consecutive. Thus exactly three matching endpoints lie on
the cycle, and all four outside vertices are matched.

Relabel the three consecutive endpoints $c_0,c_1,c_2$,
and put $F=\{c_4,\ldots,c_{12}\}$.
For $t\in T$, its matching partner in $S$ differs from those
of $c_0,c_1,c_2$. If $tc_j$ were red for some
$j\in\{4,\ldots,12\}$, an arc from $c_j$ to one of
$c_0,c_1,c_2$, preceded by $t$, would be a red path on
$9,10,11$, or $12$ vertices. This is impossible by the
path restriction. Hence all $T$--$F$ edges are blue.

In fact, every $S$--$F$ edge is blue. If $sf$ were red for
$s\in S$ and $f\in F$, at least two of the three matched
cycle vertices $c_0,c_1,c_2$ lie at cyclic distance
$3,4,5$, or $6$ from $f$. One of their matching partners
is different from $s$. The appropriate arc of the red
$C_{14}$ would then give a forbidden red path on
$9,10,11$, or $12$ vertices between $f$ and that cycle
vertex. This contradiction proves the claim.

Now each $t\in T$ is blue-adjacent to all four vertices of
$B$ and all nine of $F$. If $t$ had a blue neighbor $s\in S$,
then $B\cup\{s\}$ and any five vertices of $F$ would form
a blue $K_{5,5}$ inside the blue neighborhood of $t$:
all $B$--$F$ edges are blue, as are all $S$--$F$ edges.
This would give a blue $W_{11}$. Hence every $T$--$S$
edge is red. We already know that every $B$--$S$ edge is
red, so the red graph on $S\cup T\cup B$ contains a
$K_{7,8}$ with parts $S$ and $T\cup B$, plus an edge
inside $B$. Alternating the seven vertices of $S$ with
the eight vertices of $T\cup B$, using that edge once,
produces a red $C_{15}$. This final contradiction rules
out $(18,4)$.
\end{proof}

\begin{proposition}
\label{prop:global-eight-connected15}
Under the counterexample assumptions for $n=15$, the red graph
$G$ is at least $8$-connected.
\end{proposition}
\begin{proof}
Proposition~\ref{prop:global-six-connected15} gives
$\kappa(G)\ge7$. Suppose that a seven-vertex cut $S$ exists.
Propositions~\ref{prop:seven-cut15} and \ref{prop:seven-cut18}
leave components $A,B$ of orders
$(13,9)$ or $(11,11)$, where $A$ is the larger. Because both
components have at least five vertices, the blue-wheel
observation gives $\delta(G[A])\ge|A|-5$ and
$\Delta(\bG[B])\le4$.

Both $G[A,S]$ and $G[B,S]$ have matchings of size seven.
Indeed, if either bipartite graph had matching number at most six,
K\"onig's theorem would give a vertex cover of size at most six.
Deleting that cover would separate nonempty subsets of $A$ and
$B$, contrary to $7$-connectivity. Choose one such matching in
each graph; both saturate $S$. For each
$s\in S$, denote its matched endpoints by $a_s\in A$ and
$b_s\in B$.

The seven vertices $b_s$ cannot be independent in $G$: each
vertex of $B$ has at most four blue neighbors in $B$, so
$\alpha(G[B])\le5$. Choose distinct $s,t\in S$ with
$b_sb_t$ red. Select any $11$-vertex subset $A'$ of $A$
containing $a_s,a_t$. Because at most $|A|-11$ vertices
have been deleted,
\[
 \delta(G[A'])\ge (|A|-5)-(|A|-11)=6
              =\frac{|A'|+1}{2}.
\]
The Hamilton-connectedness degree criterion gives a red
Hamilton path through all $11$ vertices of $A'$ from
$a_s$ to $a_t$. Together with the red edges
$a_ss$, $sb_s$, $b_sb_t$, $b_tt$, and $ta_t$, it forms a
red cycle on $11+4=15$ vertices, a contradiction.
\end{proof}

\subsubsection{Excluding eight-vertex cuts}

\begin{proposition}\label{prop:next-cut-frontier}
Every eight-vertex red cut leaves a component of order at most four.
\end{proposition}
\begin{proof}
Let $S$ be an eight-vertex cut and suppose all components have
order at least five. Three components would give a blue wheel,
so there are exactly two, of orders $a\ge b\ge5$, with $a+b=21$.
By Lemma~\ref{lem:blue-sides}, the larger component has minimum
red degree at least $a-5$. If $a\ge15$, Bondy's theorem gives
$C_{15}$, while $a=14$ contradicts the core lemma.
The remaining pairs are $(13,8),(12,9),(11,10)$.
By eight-connectivity and K\"onig's theorem, both red bipartite
graphs from $S$ into the two components have matchings saturating
$S$: a vertex cover of at most seven vertices would leave nonempty
parts of both components disconnected. Two matched vertices of
the smaller component are red-adjacent, since its internal blue
maximum degree is at most four. Choose eleven vertices of the
larger component containing the two corresponding matched endpoints.
Their induced red minimum degree is at least
$(a-5)-(a-11)=6$, so the induced graph is Hamilton-connected.
A spanning red path, the two vertices of $S$, and the red edge in the
smaller component give $C_{15}$, a contradiction.
\end{proof}

\begin{proposition}\label{prop:next-cut-sizes}
Every eight-vertex red cut leaves exactly two components, of
orders $(18,3)$ or $(17,4)$.
\end{proposition}
\begin{proof}
Let $|S|=8$. Since $\delta(G)=10$, every component of $G-S$
has at least three vertices. Suppose there are at least three
components. If one component $D$ has three vertices, every
$v\in D$ has exactly ten red neighbors, namely $(D-v)\cup S$.
Its blue neighborhood consists of the other components, each of
order at least three. This contradicts
Proposition~\ref{prop:local-disconnected}.
Thus all components have order at least four. Two components of
order at least five, with a hub in a third component, give a blue
wheel. Therefore there are two components of order four.
Let $U$ be their union. Each vertex of $U$ has at least seven
red neighbors in $S$ and at least two red neighbors in its own
component. Lemma~\ref{lem:dense-eight} gives $C_{15}$.
This contradiction proves that there are exactly two components.
Their orders sum to 21, and Proposition~\ref{prop:next-cut-frontier}
requires the smaller order to be at most four. The claimed two
possibilities follow.
\end{proof}

\begin{proposition}
\label{prop:eight-cut17}
Under the $n=15$ counterexample assumptions, an eight-vertex
cut cannot leave components of orders $(17,4)$.
Consequently every eight-vertex cut, if one exists, leaves
components of orders $(18,3)$.
\end{proposition}
\begin{proof}
Suppose that $S$ is such a cut, with components $A,B$ of
orders $17,4$. Every vertex of $B$ has at least two red
neighbors in $B$ and at least seven in $S$. Thus $G[B]$
is $C_4$, $K_4$ minus one edge, or $K_4$. A vertex of
red degree two in $B$ is red-complete to $S$. Every
$s\in S$ has a red neighbor in $B$, since otherwise
$S\setminus\{s\}$ would be a cut, contrary to
Proposition~\ref{prop:global-eight-connected15}.

Lemma~\ref{lem:four-side-routing} now gives, for distinct
$s,s'\in S$ and each $q\in\{2,3,4\}$, a red
$s$--$s'$ path whose internal vertices are exactly $q$
distinct vertices of $B$.

It follows that two red $A$--$S$ edges with distinct
endpoints cannot have their $A$-endpoints joined by a
red path on $k\in\{9,10,11\}$ vertices: a red path between
the corresponding vertices of $S$ with $q=13-k\in\{4,3,2\}$
internal vertices in $B$ would complete a red $C_{15}$. There is a red matching of size eight
between $A$ and $S$, by $8$-connectivity and K\"onig's
theorem. Also $A$ contains a red $C_{12}$, since its
blue graph has no $C_{10}$ and $R(C_{12},C_{10})=16$.
Write this red cycle as $c_0c_1\cdots c_{11}c_0$.

At least three of the eight matched endpoints lie on
the cycle. Two such endpoints cannot have cyclic
distance $2,3$, or $4$, since the longer cycle arc then
has $11,10$, or $9$ vertices. Thus their permitted
distances are $1,5$, and $6$. Any set of at least three
cycle endpoints contains an adjacent pair: otherwise,
after one endpoint is labeled $c_0$, the others lie
among $c_5,c_6,c_7$, where two permitted endpoints must
be adjacent. Label an adjacent pair $c_0,c_1$. Every
other endpoint must then be $c_6$ or $c_7$. Hence there
are exactly three or four cycle endpoints.

If there are three, reflection and relabeling make them
$c_0,c_1,c_6$. All five vertices of $A\setminus V(C)$
are matched. Each is blue-adjacent to all of
$F=\{c_2,c_3,c_4,c_8,c_9,c_{10}\}$: a red edge to
$F$, followed by an arc to a matched cycle endpoint
at distance $3,4$, or $5$, would give a forbidden path
on $11,10$, or $9$ vertices. The resulting blue
$K_{5,6}$ contains a blue $C_{10}$ inside $A$, with a
vertex of $B$ as hub, a contradiction.

If there are four, they are $c_0,c_1,c_6,c_7$. Let
$T$ be the four matched vertices outside the cycle,
and put $F=V(C)\setminus\{c_0,c_1,c_6,c_7\}$.
Each vertex of $F$ is at distance $3,4$, or $5$ from
at least one matched cycle endpoint, and at distance
$2,3$, or $4$ from at least two of them. The forbidden
path orders therefore imply that all $T$--$F$ and
$S$--$F$ edges are blue; in the second assertion, use
a cycle endpoint whose matching partner differs from
the vertex of $S$ under consideration.

If $t\in T$ had a blue neighbor $s\in S$, then
$B\cup\{s\}$ and any five vertices of $F$ would give
a blue $K_{5,5}$ in the blue neighborhood of $t$.
Thus $T$ is red-complete to $S$. Put $U=B\cup T$.
Every vertex of $U$ has at least seven red neighbors
in $S$, and $G[U]$ contains a red edge in $B$.
Lemma~\ref{lem:dense-eight} supplies a red $C_{15}$,
the final contradiction. The remaining cut-size claim
follows from Proposition~\ref{prop:next-cut-sizes}.
\end{proof}

\begin{proposition}
\label{prop:global-nine-connected15}
Under the counterexample assumptions for $n=15$, the red
graph $G$ is at least $9$-connected.
\end{proposition}
\begin{proof}
By Proposition~\ref{prop:global-eight-connected15},
$\kappa(G)\ge8$. If an eight-vertex cut $S$ exists,
Proposition~\ref{prop:eight-cut17} leaves only components
$A,B$ of orders $18,3$. Every vertex of $B$ has at most
$2+8=10$ possible red neighbors, so $B$ is a red $K_3$
complete to $S$. The red $A$--$S$ bipartite graph has a
matching of size eight, by $8$-connectivity and
K\"onig's theorem. Also $A$ contains a red $C_{14}$,
because its blue graph has no $C_{10}$ and
$R(C_{14},C_{10})=18$.

For two red $A$--$S$ edges with distinct endpoints, a red
path on $k\in\{10,11,12\}$ vertices between their
$A$-endpoints would yield a red $C_{15}$ by connecting
the corresponding vertices of $S$ through $q=13-k\in\{3,2,1\}$
vertices of $B$. Thus these three path orders are forbidden.
Write the red cycle as $C=c_0c_1\cdots c_{13}c_0$.
At least four matching endpoints lie on $C$, and no two
have cyclic distance $3,4$, or $5$: the longer arc would
have $12,11$, or $10$ vertices. Permitted cyclic distances
are therefore $1,2,6$, and $7$.

These cycle endpoints include an adjacent pair. Otherwise,
after choosing one as $c_0$, the others lie in
$\{c_2,c_6,c_7,c_8,c_{12}\}$. Subject to nonadjacency,
the only permitted pairs among these five vertices are
$\{c_2,c_8\}$, $\{c_6,c_8\}$, and $\{c_6,c_{12}\}$.
No three can be chosen, contradicting the existence of
at least four endpoints in total. Relabel an adjacent
pair $c_0,c_1$. Every other endpoint lies in
$\{c_2,c_7,c_8,c_{13}\}$, whose permitted pairs are
$\{c_2,c_8\}$, $\{c_7,c_8\}$, and $\{c_7,c_{13}\}$.
Again no three can be chosen. Hence exactly four matching
endpoints lie on $C$, and all four vertices of
$T=A\setminus V(C)$ are matched.

Every edge from $c_4$ to $T$ is blue. A red edge $tc_4$
would extend the longer $c_4$--$c_0$ arc to a red path
on twelve vertices from $t$ to $c_0$, contrary to the
path restriction. Every $c_4$--$S$ edge is also blue:
if $sc_4$ were red, choose $c_0$ or $c_1$ whose matching
partner differs from $s$; the longer arc to it has
eleven or twelve vertices. Edges from $c_4$ to $B$ are
blue because $A,B$ are different red components of $G-S$.
Thus every red neighbor of $c_4$ lies on $C$.

Finally, all four chords $c_4c_1$, $c_4c_7$, $c_4c_8$,
and $c_4c_9$ are blue. If one were red, it would shortcut
the spanning $c_1$--$c_0$ arc
$c_1c_2\cdots c_{13}c_0$ to a path on twelve, twelve,
eleven, or ten vertices, respectively. Such a path is
forbidden because $c_0,c_1$ have distinct matching partners
in $S$. Consequently $c_4$ has at most $13-4=9$ red
neighbors, contradicting $\delta(G)=10$. No eight-vertex
cut exists, so $\kappa(G)\ge9$.
\end{proof}

\section{The fourteen-cycle versus the eleven-vertex wheel}
\label{sec:w11-fourteen}
We now prove the first equality in Theorem~\ref{thm:eleven}.
The counterexample assumptions in this section refer to $n=14$.
Let $v$ be a vertex of minimum red degree and put $X=N_G(v)$.
The key step is to prove $\nu(G[V(C),X])\le4$ for a red
twelve-cycle $C$ in the blue neighborhood of $v$.

By K\"onig's theorem, $G[V(C),X]$ then has a vertex cover of
size at most four. Enlarging this cover to four vertices and
adding the five vertices of the blue neighborhood outside $C$
produces a vertex cut $S$ of order nine. The two remaining vertex
sets have only four possible pairs of orders, up to interchange.
We construct a red path in each set and join the paths through
two vertices of $S$. The required path orders, both here and
in the argument for $n=15$, are listed in
Table~\ref{tab:w11-final-path-orders}.

To prove the inequality, we suppose that $G[V(C),X]$ contains
five pairwise vertex-disjoint edges and show that some vertex
of $C$ has red degree at most eight, contradicting $\delta(G)=9$.

\begin{lemma}
\label{lem:n14-local-cycle12-four-matching}
Under the counterexample assumptions for $n=14$, let $v$ have
red degree nine, put $X=N_G(v)$, and let
$H=G[N_{\overline G}(v)]$. If $C$ is a red $C_{12}$ in $H$, then
\[
\nu(G[V(C),X])\le4.
\]
\end{lemma}
\begin{proof}
Write $C=c_0c_1\cdots c_{11}c_0$. Suppose that $G[V(C),X]$
contains a matching of size five, and let $E$ be the set of its
endpoints on $C$.
No two matched endpoints can be joined by a red path on eleven
vertices inside $C$ and its chords: that path would close
through their two distinct partners in $X$ and $v$ to a red
$C_{14}$. In particular, no two indices in $E$ have cyclic
distance two.

The distance-two graph on the twelve indices consists of
two six-cycles, one on each parity class. Each contributes
at most three vertices to $E$. Since $|E|=5$, one parity
class contributes exactly three, and those three form an
alternating triple in its six-cycle. Up to rotation, this
triple is $\{c_0,c_4,c_8\}$. The other parity class contributes
at least one vertex. Rotations by four and reflection preserve
the triple and act transitively on the six odd indices, so
we may assume that $\{c_0,c_1,c_4,c_8\}\subseteq E$.

Put $w=c_2$. Its distance-two vertices $c_0,c_4$ are matched
to different vertices of $X$. Consequently every $w$--$X$
edge is blue: a proposed red edge combines with an
eleven-vertex arc to one of those two endpoints whose matching
partner differs from the proposed neighbor.

The following eight chords are also blue. For each proposed
red chord $c_2c_q$, the table supplies an eleven-vertex red
path whose endpoints belong to $\{c_0,c_1,c_4,c_8\}$,
contrary to the forbidden-path observation. Index sequences
denote the corresponding sequences of cycle vertices.
\[
\begin{array}{c|l}
q&\text{path using }c_2c_q\\\hline
0&1,2,0,11,10,9,8,7,6,5,4\\
4&0,11,10,9,8,7,6,5,4,2,1\\
5&0,11,10,9,8,7,6,5,2,3,4\\
6&1,0,11,10,9,8,7,6,2,3,4\\
7&4,5,6,7,2,1,0,11,10,9,8\\
9&0,11,10,9,2,3,4,5,6,7,8\\
10&1,0,11,10,2,3,4,5,6,7,8\\
11&0,1,2,11,10,9,8,7,6,5,4
\end{array}
\]
Thus $w$ has at most $11-8=3$ red neighbors on $C$, none in
$X\cup\{v\}$, and at most five in $V(H)\setminus V(C)$.
This gives $d_G(w)\le8$, contrary to $\delta(G)=9$.
\end{proof}

\begin{lemma}
\label{lem:nine-six-marked}
Let $A$ be a graph of order nine and minimum degree at least
four, and let six of its vertices be marked. Then $A$ has a
Hamilton path whose two endpoints are distinct marked vertices.
\end{lemma}
\begin{proof}
Add a new vertex $z$ adjacent to precisely the six marked
vertices. In the resulting ten-vertex graph, at most three
vertices have degree four, and all other vertices have degree
at least five. Its nondecreasing degree sequence satisfies
$d_i>i$ for $1\le i\le4$. Chv\'atal's sufficient
degree-sequence condition~\cite{Chvatal} therefore gives a
Hamilton cycle. Deleting $z$ leaves the required path.
\end{proof}

\begin{theorem}\label{theorem:R14=27}
With $W_{11}=K_1+C_{10}$, we have
\[
R(C_{14},W_{11})=27.
\]
\end{theorem}
\begin{proof}
The red graph $K_{13}\cup K_{13}$ on 26 vertices has no
$C_{14}$, and its bipartite blue complement has no $W_{11}$.
This proves the lower bound 27.

For the upper bound, suppose a coloring on 27 vertices has
neither a red $C_{14}$ nor a blue $W_{11}$. By
Corollary~\ref{cor:small14} and
Proposition~\ref{prop:global-eight-connected14}, its red graph
satisfies $\delta(G)=9$ and $\kappa(G)\ge8$. Choose a vertex
$v$ of red degree nine, put $X=N_G(v)$, and let
$H=G[N_{\overline G}(v)]$. The graph $H$ has seventeen
vertices and no blue $C_{10}$. Since $R(C_{12},C_{10})=16$,
it contains a red $C_{12}$, denoted $C$.

Lemma~\ref{lem:n14-local-cycle12-four-matching} and K\"onig's
theorem give a vertex cover of $G[V(C),X]$ of size at most four. Extend this cover, if necessary, to a set
$Z\subseteq V(C)\cup X$ of exactly four vertices. Put
\[
T=V(H)\setminus V(C),\quad S=T\cup Z,\quad
A=V(C)\setminus Z,\quad B=\{v\}\cup(X\setminus Z).
\]
Then $|S|=9$. If $a=|Z\cap V(C)|$, the two sides have orders
$|A|=12-a$ and $|B|=6+a$, where $0\le a\le4$.
All $A$--$B$ edges are blue, and $v$ is red-adjacent to
every other vertex of $B$. We refer to $A,B$ as sides;
we do not need either induced graph to be connected.

Every vertex of each side has at most four blue neighbors
within that side. Otherwise five such neighbors, together
with five vertices of the opposite side, give a blue
$C_{10}$ in its blue neighborhood. Consequently
\[
\delta(G[A])\ge |A|-5,\qquad
\delta(G[B])\ge |B|-5.
\]
We shall repeatedly use the fact that $\nu(G[D,S])\ge8$ whenever
$D\in\{A,B\}$ and $|D|\ge8$. Otherwise, K\"onig's theorem would
give a vertex cover of $G[D,S]$ of size at most seven. Deleting
this cover would leave a vertex of $D$ and remove every edge
joining the remaining vertices of $D$ to $S$. The other side
would be nonempty and untouched by the cover, contradicting
$\kappa(G)\ge8$.

Up to interchanging $A$ and $B$, the possible pairs $(|A|,|B|)$ are
$(12,6)$, $(11,7)$, $(10,8)$, and $(9,9)$.
In the first two cases, let $L$ be the larger side and $D$
the smaller. A red $L$--$S$ matching of size eight misses
only one vertex of $S$. Any $b\in D$ has at least
$9-(|D|-1)$, respectively four or three, red neighbors
in $S$. Hence two of these, say $s,t$, occur in the matching.
Let their distinct matched neighbors in $L$ be $a_s,a_t$.
Choose an eleven-vertex set $L'\subseteq L$ containing
$a_s,a_t$. Its minimum red degree is at least
\[
(|L|-5)-(|L|-11)=6,
\]
so $G[L']$ is Hamilton-connected. A spanning red $a_s$--$a_t$
path in $G[L']$ closes through $sbt$ to a red $C_{14}$.

Now suppose the two sides have orders ten and eight, and denote
them by $L,D$, respectively. Choose a matching of size eight in each of
$G[L,S]$ and $G[D,S]$. At least seven vertices of $S$ are
saturated by both matchings. Choose seven of them and index
the corresponding matching edges as
\[
a_i s_i,\ s_i b_i\in E(G),\qquad 1\le i\le7,
\]
with distinct $a_i\in L$ and distinct $b_i\in D$.
The graph on $D$ has minimum degree at least three. Every
marked $b_i$ has at least two marked red neighbors, because
only one vertex of $D$ is unmarked. Thus, if $G[L]$ is
Hamilton-connected, a spanning red path between suitable
$a_i,a_j$ closes through the red edge $b_i b_j$ and
$s_i,s_j$ to a red cycle of order $10+2+2=14$.

The graph on $L$ has minimum degree at least five, so it
satisfies Ore's degree-sum condition on ten vertices.
If $G[L]$ is not Hamilton-connected, Theorem~3 of
Shih et al.~\cite{SSK} implies that it has one of the forms
\[
H_2\vee(K_4\cup K_4)
\quad\hbox{or}\quad H_5\vee\overline K_5.
\]
For the first form, any vertex outside the two joining
vertices is the endpoint of a Hamilton path ending at any
other prescribed vertex, as explicitly constructed in
Lemma~\ref{lem:two-dirac-components}. Choose a marked $a_i$
outside those two vertices and a marked red neighbor $b_j$
of $b_i$. The same cycle construction gives $C_{14}$.

For the second form, the independent set of five vertices
and Lemma~\ref{lem:independent-five-separated} imply that
$\cl G[D]$ has at most one edge. Hence $G[D]$ is
Hamilton-connected. The seven marked vertices of $L$ meet
both parts of the spanning red $K_{5,5}$. Choose $a_i,a_j$
in opposite parts and join them by a red path on four
vertices in this $K_{5,5}$. A spanning red $b_i$--$b_j$
path in $D$, together with this path and $s_i,s_j$, forms
a red cycle on $4+8+2=14$ vertices. This excludes the
ten--eight distribution in both orientations.

Finally suppose $|A|=|B|=9$. Here we keep the original
names so that $v\in B$ is red-adjacent to all of $B-v$.
Choose matchings of size eight in $G[A,S]$ and $G[B,S]$.
Again, at least seven vertices of $S$ are saturated by both.
Discard the at most one pair whose $B$-endpoint is $v$,
and retain six pairs, written
\[
a_i s_i,\ s_i b_i\in E(G),\qquad 1\le i\le6,
\]
where $b_i\ne v$. The red graph on $A$ has minimum
degree at least four. By Lemma~\ref{lem:nine-six-marked},
it has a Hamilton path with endpoints $a_i,a_j$ for some
$i\ne j$. The red path $b_i v b_j$ in $B$ and the vertices
$s_i,s_j\in S$ close this path into a red cycle
on $9+3+2=14$ vertices. All possible distributions have
been excluded, proving the upper bound.
\end{proof}

\section{The fifteen-cycle versus the eleven-vertex wheel}
\label{sec:w11-fifteen}
We prove the second equality in Theorem~\ref{thm:eleven}.
The counterexample assumptions in this section refer to $n=15$.
Let $v$ be a vertex of minimum red degree, put $X=N_G(v)$, and
let $C$ be a red fourteen-cycle in the blue neighborhood of $v$.
We prove $\nu(G[V(C),X])=5$. By K\"onig's theorem, $G[V(C),X]$
has a minimum vertex cover of size five. Together with the four
vertices of the blue neighborhood outside $C$, this cover gives
a vertex cut $S$ of order nine.

Once the orders of the components of $G-S$ are determined, the
final step again reduces to finding two red paths whose orders
sum to $n-2$. Joining these paths through two vertices of $S$
completes the required cycle. Table~\ref{tab:w11-final-path-orders}
summarizes the required path orders. The proofs of
Theorems~\ref{theorem:R14=27} and~\ref{theorem:R15=29} establish
the existence and the required endpoints of these paths.

\begin{table}[tbp]
\centering
\small
\caption{Path orders in the final separations. Each construction
uses one path from each side and two distinct vertices of $S$.
The required paths and endpoints are constructed in the proofs.}
\label{tab:w11-final-path-orders}
\begin{tabular}{@{}c c c c@{}}
\toprule
$n$ & Side orders & Red paths in the two sides & Cycle order \\
\midrule
14 & $(12,6)$ & $P_{11},P_1$ & $11+1+2=14$ \\
14 & $(11,7)$ & $P_{11},P_1$ & $11+1+2=14$ \\
14 & $(10,8)$ & $P_{10},P_2$ or $P_4,P_8$
& $10+2+2=4+8+2=14$ \\
14 & $(9,9)$ & $P_9,P_3$ & $9+3+2=14$ \\
\midrule
15 & $(13,7)$ & $P_{12},P_1$ & $12+1+2=15$ \\
15 & $(12,8)$ & $P_{12},P_1$ & $12+1+2=15$ \\
15 & $(11,9)$ & $P_{11},P_2$ & $11+2+2=15$ \\
15 & $(10,10)$ & $P_{10},P_3$ & $10+3+2=15$ \\
\bottomrule
\end{tabular}

\smallskip
\parbox{0.94\linewidth}{\footnotesize
$P_1$ denotes a single vertex. The paths are listed in the order
of the two side sizes; equal-sized sides may be interchanged.
For the $(9,9)$ case at $n=14$, the side supplying $P_3$ contains
the original minimum-degree vertex, which is adjacent to the
other vertices of that side.}
\end{table}

\subsection{Edges between the local cycle and the red neighborhood}
The inequality $\nu(G[V(C),X])\ge5$ follows from
$\kappa(G)\ge9$. To prove the reverse inequality, we suppose that
$G[V(C),X]$ contains a matching of size six. We distinguish two
cases according to the indices of its endpoints on $C$. If these
indices have the same parity, we obtain the blue-clique
configuration in the next lemma. Otherwise, the possible endpoint
configurations force a vertex of red degree at most nine. Both
cases contradict the counterexample assumptions.

\begin{lemma}
\label{lem:seven-clique-two-exceptions}
Let the edges of a complete graph with vertex partition $F\cup D$
be colored red and blue, where $|F|=7$ and $|D|=11$.
Suppose that there is no blue $C_{10}$ and that
$F$ is a blue clique, every vertex of $F$ has at most two blue
neighbors in $D$, and at most two vertices of $F$ have more than
one blue neighbor in $D$. Then the graph contains a red $C_{15}$.
\end{lemma}
\begin{proof}
There are at most $5+2\cdot2=9$ blue edges between $F$ and $D$.
Consequently at most one vertex of $D$ has no red neighbor in $F$.
The set $D_+$ of vertices having a red neighbor in $F$ has order
at least ten, so it contains a red edge $ab$; otherwise it is a
blue clique containing $C_{10}$. We can choose distinct red
neighbors $f_a,f_b\in F$ of $a,b$. Indeed, failure would require
both red neighborhoods in $F$ to be the same singleton, giving
at least twelve blue $F$--$D$ edges.

Order the seven vertices of $F$ starting at $f_a$ and ending at
$f_b$, with the at most two exceptional vertices nonconsecutive.
Such an order always exists: if neither exception is prescribed
at an end, place them in positions two and five; if exactly one
is prescribed at an end, place the other in position four; and
if both are prescribed at the ends, they are already separated.
The cases with fewer than two exceptions are immediate.
Every consecutive pair in this order has at least
$9-(1+2)=6$ common red neighbors in $D\setminus\{a,b\}$.
Choose distinct common neighbors for the six consecutive pairs,
one at a time. They form, together with $F$ and $a,b$, a red
alternating $a$--$b$ path on fifteen vertices. Closing it with
$ab$ gives the desired red $C_{15}$.
\end{proof}

\begin{proposition}\label{prop:local-cycle14-five-matching}
Under the counterexample assumptions for $n=15$, let $v$ have
minimum red degree, put $X=N_G(v)$, and let
$H=G[N_{\overline G}(v)]$. For every red cycle $C$ of order
fourteen in $H$,
\[
\nu(G[V(C),X])=5.
\]
\end{proposition}
\begin{proof}
Write $C=c_0c_1\cdots c_{13}c_0$ and $T=V(H)\setminus V(C)$,
so $|T|=4$ and $|X|=10$. If $\nu(G[V(C),X])\le4$,
K\"onig's theorem would give a vertex cover of $G[V(C),X]$
$Z\subseteq V(C)\cup X$ of size at most four. Removing $Z\cup T$
would separate $V(C)\setminus Z$ from $v$: all $C$--$X$ red
edges meet $Z$, and all $C$--$v$ edges are blue. Both sides are
nonempty, because $|C|=14>|Z|$ and $v\notin Z\cup T$.
This contradicts $\kappa(G)\ge9$. Hence $\nu(G[V(C),X])\ge5$.

For red attachment edges $c_i x,c_j y$ with $i\ne j$ and $x\ne y$,
a red path on
twelve vertices from $c_i$ to $c_j$ inside $C$ and its chords
would close through $xvy$ to a red $C_{15}$. We use this
forbidden-path observation throughout the proof.

Suppose $G[V(C),X]$ contains a matching of size six, and let $E$ be its
six endpoints on $C$. No two indices in $E$ have cyclic distance
three, because their longer cycle arc has twelve vertices.
The graph on the fourteen indices whose edges join indices
differing by three modulo fourteen is itself a fourteen-cycle.
In its cyclic order, the six gaps between successive members
of $E$ contain positive numbers of unselected indices summing
to eight. Thus the gap sizes are either $3,1,1,1,1,1$ or a
permutation of $2,2,1,1,1,1$. In the first case all six
members of $E$ have the same parity in the original cycle.
In the second case the two gaps of size two split the selected
vertices into two strings of lengths $k$ and $6-k$, where, up
to reflection, $k=1,2$, or $3$. Up to rotation and reflection
of $C$, the three mixed-parity possibilities are therefore
\[
\{0,1,2,6,8,10\},\qquad
\{0,1,2,6,7,8\},\qquad
\{0,1,2,7,8,9\}.
\]
These representatives can also be verified directly by listing
their gap sizes in the index order $0,3,6,9,12,1,4,7,10,13,2,5,8,11$.

First suppose the six endpoints have the same parity; relabel
so that six of the seven even-indexed vertices are in $E$.
Put $F=\{c_1,c_3,\ldots,c_{13}\}$.
Every odd-indexed vertex having two members of $E$ at distance
three has no red neighbor in $X$, by the forbidden-path
observation and the distinct matching partners of those two
members. The other two odd-indexed vertices, each having just
one member of $E$ at distance three, have at most one red
neighbor in $X$, namely that member's matching partner.

We claim that $F$ is independent in $G$. Up to rotation and
reflection, a proposed red chord is $c_1c_j$ for $j=3,5$, or $7$.
For each value of $j$, the following two twelve-vertex paths
have disjoint even-indexed endpoint pairs. Since only one
even-indexed vertex is absent from $E$, one pair lies in $E$,
contradicting the forbidden-path observation:
\[
\begin{array}{c|l}
3&0,1,3,4,5,6,7,8,9,10,11,12\\
 &2,3,1,0,13,12,11,10,9,8,7,6\\\hline
5&0,13,12,11,10,9,8,7,6,5,1,2\\
 &4,5,1,0,13,12,11,10,9,8,7,6\\\hline
7&0,13,12,11,10,9,8,7,1,2,3,4\\
 &2,3,4,5,6,7,1,0,13,12,11,10
\end{array}
\]
Here and below an index sequence denotes the corresponding
sequence of cycle vertices.

Let $D=V(H)\setminus F$, of order eleven. The five vertices
of $F$ having no red neighbor in $X$ have at least ten red
neighbors in $D$, by $\delta(G)=10$. The other two have at
least nine red neighbors in $D$. All edges from $v$ to $F$
are blue. Lemma~\ref{lem:seven-clique-two-exceptions} therefore
applies to $H$, giving a red $C_{15}$, a contradiction.

It remains to treat the three mixed-parity configurations.
The first two contain $\{c_0,c_1,c_2,c_6,c_8\}$. Put $w=c_3$.
Its distance-three neighbors $c_0,c_6$ are both matched, so all
$w$--$X$ edges are blue. The following table shows that the
eight chords from $w$ to $c_q$, for
$q\in\{0,1,5,6,7,8,9,10\}$, are blue as well: if a chord
were red, its listed twelve-vertex path would have both
endpoints in $\{c_0,c_1,c_2,c_6,c_8\}\subseteq E$.
\[
\begin{array}{c|l}
q&\text{path using }c_3c_q\\\hline
0&1,2,3,0,13,12,11,10,9,8,7,6\\
1&2,3,1,0,13,12,11,10,9,8,7,6\\
5&0,13,12,11,10,9,8,7,6,5,3,2\\
6&0,13,12,11,10,9,8,7,6,3,2,1\\
7&0,13,12,11,10,9,8,7,3,4,5,6\\
8&1,0,13,12,11,10,9,8,3,4,5,6\\
9&0,13,12,11,10,9,3,4,5,6,7,8\\
10&1,0,13,12,11,10,3,4,5,6,7,8
\end{array}
\]
Thus $w$ has at most $13-8=5$ red neighbors on $C$, at most
four in $T$, and none in $X\cup\{v\}$. Its total red degree
is at most nine, contrary to $\delta(G)=10$.

For the last configuration $E=\{c_0,c_1,c_2,c_7,c_8,c_9\}$,
put $w=c_4$. Its distance-three neighbors $c_1,c_7$ are matched,
so again all $w$--$X$ edges are blue. Eight blue chords are
forced by the following table, giving the same degree contradiction:
\[
\begin{array}{c|l}
q&\text{path using }c_4c_q\\\hline
0&1,2,3,4,0,13,12,11,10,9,8,7\\
1&0,13,12,11,10,9,8,7,6,5,4,1\\
2&0,13,12,11,10,9,8,7,6,5,4,2\\
6&0,13,12,11,10,9,8,7,6,4,3,2\\
7&0,13,12,11,10,9,8,7,4,3,2,1\\
8&1,0,13,12,11,10,9,8,4,5,6,7\\
9&1,0,13,12,11,10,9,4,5,6,7,8\\
10&1,0,13,12,11,10,4,5,6,7,8,9
\end{array}
\]
Thus $G[V(C),X]$ has no matching of size six, so
$\nu(G[V(C),X])\le5$. Together with the reverse inequality
proved above, this gives $\nu(G[V(C),X])=5$.
\end{proof}

\begin{corollary}
\label{cor:local-cycle14-nine-cut}
Under the preceding assumptions, $\kappa(G)=9$. For any such
cycle $C$, a minimum vertex cover $Z$ of $G[V(C),X]$
has size five. Put $T=V(H)\setminus V(C)$,
$a=|Z\cap V(C)|$, and $S=T\cup Z$.
Then $1\le a\le5$, and $G-S$ has exactly two red components
\[
A=V(C)\setminus Z,\qquad B=\{v\}\cup(X\setminus Z),
\]
of orders $14-a$ and $6+a$. Every vertex of either component
has at most four blue neighbors within that component, and
$v$ is red-adjacent to every other vertex of $B$.
\end{corollary}
\begin{proof}
K\"onig's theorem and Proposition~\ref{prop:local-cycle14-five-matching}
give $|Z|=5$. The set $S$ has order nine and separates the
nonempty sets $A$ and $B$, proving $\kappa(G)=9$. We have
$0\le a\le5$, $|A|=14-a\ge9$, and $|B|=6+a\ge6$.
All $A$--$B$ edges are blue. If a vertex in either side had
five blue neighbors within its own side, those five vertices
and five vertices of the other side would form a blue $C_{10}$
in its blue neighborhood. This proves the internal blue-degree
bound. The graph $G[B]$ is connected because $v$ is red-adjacent
to all of $X\setminus Z$. The red graph on $A$ is connected
as well: if it had two components, each would have at least
$|A|-4$ vertices, impossible when $|A|\ge9$.

Finally, if $a=0$, then $A=V(C)$ has fourteen vertices and
$\delta(G[A])\ge14-5=9$. Thus $G[A]$ is Hamilton-connected by the
minimum-degree criterion, contradicting Lemma~\ref{lem:core}.
Thus $a\ge1$.
\end{proof}

\subsection{The final nine-vertex cut}
The cycle constructions below have a common form. Suppose $P_A$
joins $a_s$ to $a_t$ in one side, $P_B$ joins $b_s$ to $b_t$ in
the other, and $s,t$ are distinct vertices of $S$ with red edges
$a_ss,sb_s,a_tt,tb_t$. The path $P_B$ may be a single vertex,
in which case $b_s=b_t$. The two paths and the attachment edges
form a red cycle, with
\[
 |V(C)|=|V(P_A)|+|V(P_B)|+2.
\]
The remaining issue in each case is to choose paths of the required
orders with these attachment endpoints.

\begin{theorem}\label{theorem:R15=29}
With the convention $W_{11}=K_1+C_{10}$, we have
\[
R(C_{15},W_{11})=29.
\]
\end{theorem}
\begin{proof}
The disjoint union $K_{14}\cup K_{14}$ has no red $C_{15}$,
and its blue complement is bipartite and therefore contains
no wheel $K_1+C_{10}$. This gives the lower bound 29.

Suppose that a coloring on 29 vertices has neither a red
$C_{15}$ nor a blue $W_{11}$. The preceding reductions apply:
Corollary~\ref{cor:small} gives $\delta(G)=10$;
Proposition~\ref{prop:global-nine-connected15} gives
$\kappa(G)\ge9$.
Choose a vertex $v$ of red degree ten. Its blue neighborhood
$H$ has eighteen vertices and no blue $C_{10}$. Thus
$R(C_{14},C_{10})=18$ supplies a red $C_{14}$ in $H$.
Proposition~\ref{prop:local-cycle14-five-matching} and
Corollary~\ref{cor:local-cycle14-nine-cut} now give a red
vertex cut $S$ of order nine, whose deletion leaves two
components. Up to exchanging these components, their orders
are
\[
(13,7),\quad(12,8),\quad(11,9),\quad(10,10).
\]
Denote the two components by $A$ and $B$, with $|A|\ge|B|$.
Every vertex has at most four blue neighbors within its own
component, so
\[
\delta(G[A])\ge |A|-5,\qquad
\delta(G[B])\ge |B|-5.
\]

We shall repeatedly use the fact that $G[D,S]$ has a matching
saturating $S$ whenever $D\in\{A,B\}$ and $|D|\ge9$. Otherwise,
K\"onig's theorem would give a vertex cover of $G[D,S]$ of size
at most eight. Deleting this cover would leave a vertex of $D$
and remove every edge joining the remaining vertices of $D$ to
$S$. The other component would be nonempty and untouched by the
cover, contradicting $\kappa(G)\ge9$.
Write $a_s$ for the distinct matched neighbor in $A$ of each
$s\in S$ whenever this matching is used.

Suppose first that $(|A|,|B|)$ is $(13,7)$ or $(12,8)$.
Any $b\in B$ has at least $10-(|B|-1)\ge3$ red neighbors
in $S$. Choose distinct such neighbors $s,t$, and choose
a twelve-vertex subset $A'\subseteq A$ containing both
$a_s$ and $a_t$. Its minimum red degree is at least
\[
(|A|-5)-(|A|-12)=7,
\]
so $G[A']$ is Hamilton-connected. A spanning red $a_s$--$a_t$
path in $G[A']$ closes through $sbt$ to a red cycle on
$12+3=15$ vertices.

If $(|A|,|B|)=(11,9)$, the graph on $A$ has minimum degree
at least six and is Hamilton-connected. The graph on $B$ has
minimum degree at least four, so choose a red edge $bb'$ in it.
Both $b$ and $b'$ have at least $10-(9-1)=2$ red neighbors in
$S$. Choose distinct $s,t\in S$ with $sb$ and $b't$ red.
A spanning red $a_s$--$a_t$ path in $A$ closes through
$sbb't$ to a red cycle on $11+4=15$ vertices.

It remains to exclude $(|A|,|B|)=(10,10)$. Both induced red
graphs have minimum degree at least five. Matchings saturating
$S$ on both sides can be indexed as
\[
a_i s_i,\ s_i b_i\in E(G),\qquad 1\le i\le9,
\]
where the $a_i$ are distinct and the $b_i$ are distinct.
For every $i$, there is some $j\ne i$ and a red path
$b_i z b_j$ on three vertices in $B$. To see this, choose
any red neighbor $z$ of $b_i$. Among its at least five red
neighbors, at most one is outside the nine marked vertices
$b_1,\ldots,b_9$ and at most one is $b_i$; hence at least
three choices of $b_j$ remain. Whenever $G[A]$ has a spanning
$a_i$--$a_j$ path, this path, the vertices $s_i,s_j\in S$,
and $b_i z b_j$ form a red $C_{15}$.
The same observation holds with $A$ and $B$ interchanged.

In particular, neither $G[A]$ nor $G[B]$ is Hamilton-connected.
Since both graphs satisfy Ore's degree-sum condition on ten vertices,
Theorem~3 of Shih et al.~\cite{SSK} leaves only the forms
\[
H_2\vee(K_4\cup K_4)
\qquad\hbox{or}\qquad
H_5\vee\overline K_5.
\]
In the first form, the two clique orders must both equal four
because their vertices have degree at least five.

Suppose, for example, that $A=H_2\vee(K_4\cup K_4)$.
Every vertex outside the two-vertex set $H_2$ is an endpoint
of a Hamilton path whose other endpoint is any prescribed
different vertex. Here is an explicit construction. Name the
two joining vertices $p,q$, and the cliques $L,R$, with the
first prescribed endpoint $a\in L$.
If the other endpoint $b$ lies in $L$, split $L$ into two
nonempty paths starting at $a$ and ending at $b$, respectively,
and concatenate them as
\[
L_1,\ p,\ R,\ q,\ L_2.
\]
If $b\in R$, choose $r\in R\setminus\{b\}$ and use
\[
L,\ p,\ r,\ q,\ R\setminus\{r\},
\]
ordering the clique paths to start at $a$ and end at $b$.
If $b=p$, use $L,q,R,p$; the case $b=q$ is symmetric.
These constructions never require the edge $pq$.
Choose a marked $a_i$ outside $H_2$, which is possible because
nine vertices are marked. The preceding three-vertex-path
observation in $B$ supplies $j\ne i$, and the stated Hamilton
path in $A$ then gives a red $C_{15}$. Thus the first form is
impossible on either side.

The remaining form on $A$ contains five mutually blue-adjacent
vertices. Lemma~\ref{lem:independent-five-separated} therefore
forces the blue graph on the other ten-vertex component $B$
to have at most one edge. Hence $\delta(G[B])\ge8$, so $B$
is Hamilton-connected, contrary to the case already excluded.
All four component distributions are impossible, completing
the upper bound.
\end{proof}

\section{The seven-vertex wheel}\label{sec:wheel-seven}
\subsection{The critical graphs for \texorpdfstring{$(C_8,C_6)$}{(C8,C6)}}
We show that every graph in $\Hc_8$ contains a Hamilton-connected
induced subgraph on seven vertices.
\begin{lemma}\label{lem:red7}
  Every graph in $\Hc_8$ contains a red $C_7$.
\end{lemma}
\begin{proof}
  Lemma 3.1 of~\cite{kr}, with parameters $n=7,k=6$, gives a monochromatic
  cycle of length at least seven, since a blue $C_6$ is forbidden.
  A monochromatic $C_9$ gives a monochromatic $C_8$ by Lemma 2.1(1)
  of the same paper. A red $C_8$ is forbidden.
  If there is a blue $C_8$, all its chords joining vertices at cyclic
  distance three must be red, since any blue chord of this type closes a blue
  six-cycle along the longer arc. These chords form a red $C_8$, because
  $\gcd(3,8)=1$. Thus a monochromatic $C_7$ exists.
  If it is blue, all its chords of cyclic distance two are red; these form
  a red $C_7$ because $\gcd(2,7)=1$.
\end{proof}

\begin{lemma}\label{lem:attach7}
  Let $H\in\Hc_8$, let $C=01234560$ be a red cycle, and write
  $V(H)\setminus V(C)=\{p,q\}$. Each of $p,q$ has at most one red neighbor
  on $C$.
\end{lemma}
\begin{proof}
  Write $N_C(t)$ for the red neighbors of $t$ on $C$, with indices
  taken modulo seven.
  By Lemma~\ref{lem:successors}, $N_C(t)$ is an independent set in the cycle $C$ and has
  size at most three. The successor and predecessor rules give
  \begin{align}
    i,i+2\in N_C(t)&\ \Longrightarrow\ (i+1)(i+3),\ (i+1)(i-1)
    \text{ are blue},\label{eq:rules7a}\\
    i,i+3\in N_C(t)&\ \Longrightarrow\ (i+1)(i+4),\ (i+2)(i-1)
    \text{ are blue}.\label{eq:rules7b}
  \end{align}
  The second edge in~\eqref{eq:rules7b} joins the two predecessors.

  Suppose first that $|N_C(p)|=3$. Up to rotation and reflection,
  $N_C(p)=\{0,2,4\}$. The forced blue edges are
  $13,16,35,15,36$. Thus on $S=\{1,3,5,6\}$ the blue graph contains
  $K_4-56$, and $p$ is blue-adjacent to all of $S$.
  The blue graph on $S\cup\{p\}$ is Hamilton-connected by
  Lemma~\ref{lem:hc}. Hence $q$ has at most one blue neighbor in $S$;
  otherwise, a spanning path on these five vertices closes through $q$
  to a blue $C_6$. Independence on the red cycle forces
  $N_C(q)=\{1,3,5\}$ or $\{1,3,6\}$. The reflection $i\mapsto4-i$
  interchanges these cases while preserving $N_C(p)$, so take the first.
  This forces $02,04,24,26,46$ to be blue.

  Set $A=\{1,3,5\}$ and $B=\{0,2,4\}$. All edges between $A$ and $B$ are
  red: the cycle supplies five of them, and if one of the remaining
  four were blue, the corresponding row below would be a blue six-cycle.
  \[
  \begin{array}{c|l}
    \text{blue edge}&\text{blue cycle}\\\hline
    14 & p,3,1,4,2,6,p\\
    03 & p,1,3,0,2,6,p\\
    05 & p,1,5,0,2,6,p\\
    25 & p,1,5,2,4,6,p
  \end{array}
  \]
  If $pq$ is blue, then $p,1,3,6,2,q,p$ is a blue $C_6$.
  If $pq$ is red, the red graph between $A\cup\{p\}$ and $B\cup\{q\}$
  is $K_{4,4}$, which contains $C_8$. Thus both outside vertices have at
  most two red neighbors on $C$.

  Suppose next that $N_C(p)=\{0,2\}$. The edges $13,16$ are blue.
  Choose $r\in\{4,5\}$ with $qr$ blue; such a choice exists since $q$
  cannot have consecutive red neighbors. Also $pr$ is blue.
  If $q3$ is blue, use $q,3,1,6,p,r,q$; if $q6$ is blue, use
  $q,6,1,3,p,r,q$. Therefore $N_C(q)=\{3,6\}$.
  This forces $04,25$ to be blue, giving $p,4,0,q,2,5,p$.

  Finally suppose $N_C(p)=\{0,3\}$. The edges $14,26$ are blue, and $p$
  is blue-adjacent to all four endpoints. If $q$ has a blue neighbor on
  each edge, orient them to obtain two internally disjoint blue three-edge
  $p$--$q$ paths, forming a blue $C_6$.
  Thus $N_C(q)=\{1,4\}$ or $\{2,6\}$. Applying the reflection
  $i\mapsto3-i$ if necessary, assume $N_C(q)=\{1,4\}$, which forces
  $25,03$ to be blue.
  Put $A=\{0,3\}$, $B=\{1,4\}$ and $D=\{2,5,6\}$.
  The edges $03,14,25,26$ are blue, $p$ is blue-adjacent to $B\cup D$,
  and $q$ is blue-adjacent to $A\cup D$.
  If $a\in A,b\in B$ and $ab$ is blue, use $p,b,a,q,5,2,p$.
  If $a\in A,d\in D$ and $ad$ is blue, let $a'$ be the other vertex of
  $A$ and take $d'\in D\setminus\{d\}$; use $p,d,a,a',q,d',p$.
  The case of a blue edge between $B$ and $D$ is symmetric.
  All edges between distinct blocks are therefore red.
  The only remaining edge within the blocks is $56$, which is red
  because it lies on $C$.
  The blue graph on $V(C)$ consequently has maximum degree two.
  The red graph has minimum degree four and is Hamilton-connected by
  Lemma~\ref{lem:hc}. A spanning red $0$--$3$ path closed through $p$
  gives $C_8$.
  These cases exhaust the possible cyclic distances between two red neighbors.
\end{proof}

Sparse red attachments also restrict the blue degrees on the core.
\begin{lemma}\label{lem:densecore}
  Let $X$ be a vertex set of size at least seven, and let $p,q$
  be distinct vertices outside $X$.
  Suppose the blue graph contains no $C_6$ and each of $p,q$ has at most
  one red neighbor in $X$. Then the blue graph on $X$ has maximum degree
  at most two.
\end{lemma}
\begin{proof}
  If $x\in X$ has three blue neighbors in $X$, choose two distinct
  neighbors $u,w$ with $pu,qw$ blue. Each of $p,q$ excludes at most one of
  the three choices, so this is possible.
  Among the at least four vertices in $X\setminus\{x,u,w\}$, at most
  two fail to be common blue neighbors of $p,q$.
  Choose a common blue neighbor $r$. Then $x,u,p,r,q,w,x$ is a blue $C_6$.
\end{proof}

\begin{proposition}\label{prop:core8}
  Every graph in $\Hc_8$ contains a Hamilton-connected induced subgraph
  on seven vertices.
\end{proposition}
\begin{proof}
  Choose a red seven-cycle by Lemma~\ref{lem:red7}.
  Lemmas~\ref{lem:attach7} and~\ref{lem:densecore} give red minimum degree
  at least four on its vertex set. Apply Lemma~\ref{lem:hc}.
\end{proof}

\subsection{The critical graphs for \texorpdfstring{$(C_9,C_6)$}{(C9,C6)}}
For $\Hc_9$, the same type of core exists apart from two bipartite
exceptions.
\begin{lemma}\label{lem:blue5}
  If $H\in\Hc_9$ contains a blue $K_5$, then
  $H\cong K_{5,5}$ or $H\cong K_{5,5}-e$.
\end{lemma}
\begin{proof}
  Let $Q$ induce a blue $K_5$ and put $P=V(H)\setminus Q$, so $|P|=5$.
  Each $p\in P$ has at most one blue neighbor in $Q$, since two such
  neighbors can be joined by a spanning blue path on $Q$ and closed through
  $p$ to form $C_6$.

  Suppose $uv$ is a red edge in $P$. Order $P$ as
  $p_1=u,p_2,p_3,p_4,p_5=v$, and for $1\le i\le4$ put
  $S_i=N_H(p_i)\cap N_H(p_{i+1})\cap Q$.
  Each $S_i$ has size at least three. Hall's condition for these four sets
  can fail only when their union has size at most three.
  In that case all four equal the same three-element set $Q\setminus\{a,b\}$,
  and the unique blue neighbors in $Q$ of $p_1,\ldots,p_5$ must alternate
  $a,b,a,b,a$, after interchanging $a,b$ if necessary.
  Interchanging $p_2,p_3$ makes one of the new sets have size four, while
  all four still have size at least three. Hall's condition now holds.
  Thus, with a suitable ordering, choose distinct representatives
  $s_i\in S_i$. The red path
  $p_1s_1p_2s_2p_3s_3p_4s_4p_5$, closed by $uv$, is a red $C_9$.
  Hence $P$ is also a blue clique.

  There is at most one blue edge between $P$ and $Q$.
  Indeed, two such edges sharing an endpoint form a blue $C_6$ with a spanning path
  in the opposite five-clique. Two disjoint such edges form a blue $C_6$
  with a three-vertex path in each clique.
  The assertion follows.
\end{proof}

\begin{lemma}\label{lem:red8}
  Every nonbipartite graph in $\Hc_9$ contains a red $C_8$.
\end{lemma}
\begin{proof}
  Lemma 3.1 of~\cite{kr}, with $n=8,k=6$, gives a monochromatic cycle
  of length at least eight.
  There is no monochromatic $C_9$: a red one is forbidden, and in a blue
  one all chords of cyclic distance four must be red to avoid a blue $C_6$;
  these chords form a red $C_9$.
  If there is a red $C_{10}$, Lemma 2.1(3) of~\cite{kr}, with the colors
  interchanged, shows that the two parity classes of this cycle are blue
  cliques. Since $|H|=10$, the cycle is spanning, so these two classes
  partition $V(H)$ into two independent sets in the red graph.
  Thus $H$ is bipartite, a contradiction.
  If there is a blue $C_{10}$, Lemma 2.1(2) of~\cite{kr} gives a
  monochromatic $C_8$.
  Finally a blue $C_8$ has all its distance-three chords red and hence
  gives a red $C_8$, as in Lemma~\ref{lem:red7}.
\end{proof}

\begin{lemma}\label{lem:attach8}
  Let $H\in\Hc_9$ be nonbipartite, let $C=012345670$ be a red cycle,
  and write $V(H)\setminus V(C)=\{p,q\}$.
  Each outside vertex has at most one red neighbor on $C$.
\end{lemma}
\begin{proof}
  Write $N_C(t)$ for the red neighbors of $t$ on $C$.
  Indices in this proof are modulo eight. Lemma~\ref{lem:successors}
  shows that each $N_C(t)$ is an independent set in the cycle $C$ and has size at most four.
  Four red neighbors would be one parity class of $C$; their successors
  would form a blue $K_4$, all blue-adjacent to $t$.
  This gives a blue $K_5$, so Lemma~\ref{lem:blue5} implies that
  $H$ is bipartite, a contradiction. Thus $|N_C(t)|\le3$.

  Up to the symmetries of $C$, an independent three-set is either
  $\{0,2,5\}$ or $\{0,2,4\}$.
  If $N_C(p)=\{0,2,5\}$, its successors $\{1,3,6\}$ and predecessors
  $\{7,1,4\}$ are blue triangles. Hence $p,3,6,1,4,7,p$ is a blue $C_6$.
  If $N_C(p)=\{0,2,4\}$, the blue graph on $S=\{1,3,5,7\}$ contains
  $K_4-57$, and $p$ is blue-adjacent to $S$.
  The blue graph on $S\cup\{p\}$ is Hamilton-connected by
  Lemma~\ref{lem:hc}.
  Vertex $q$ can have at most one blue neighbor in these five vertices,
  since otherwise it closes a spanning blue path to $C_6$.
  Since $q$ has at most three red neighbors on $C$, it follows that $pq$
  is red and $N_C(q)$ consists of precisely three vertices of $S$.
  Let $b$ be the remaining vertex of $S$, so $qb$ is blue.
  Any three-element subset of $S$ is a rotation of $\{1,3,5\}$.
  Its predecessor and successor triples are therefore two blue triangles
  on $\{0,2,4,6\}$ sharing an edge. Hence the blue graph on this set
  contains $K_4$ with at most one edge deleted.
  Choose a blue neighbor $r$ of $6$ in that set and a blue neighbor $s$ of
  $b$ in $S$. The vertices in the sequence $p,6,r,q,b,s,p$ are distinct,
  and all its edges are blue, a contradiction.
  Thus each outside vertex has at most two red neighbors on $C$.

  It remains to exclude two red neighbors, whose cyclic distance is
  two, three or four.
  If $N_C(p)=\{0,2\}$, the edges $13,17$ are blue.
  Choose $r\in\{4,5,6\}$ with $qr$ blue; also $pr$ is blue.
  If $q3$ or $q7$ is blue, use respectively
  $q,3,1,7,p,r,q$ or $q,7,1,3,p,r,q$.
  Thus $N_C(q)=\{3,7\}$, forcing $04,26$ blue.
  Now $p,4,0,q,2,6,p$ is a blue $C_6$.

  If $N_C(p)=\{0,3\}$, the edges $14,27$ are blue and $p$ is
  blue-adjacent to their four endpoints.
  To prevent two internally disjoint blue three-edge $p$--$q$ paths,
  $q$ must be red-adjacent to both endpoints of one of these edges.
  Applying the reflection $i\mapsto3-i$ if necessary, assume
  $N_C(q)=\{1,4\}$.
  This forces $25,03$ blue, and $p,5,2,7,q,6,p$ is a blue $C_6$.

  If $N_C(p)=\{0,4\}$, the same argument applied to the blue edges
  $15,37$ allows us, by reflection, to take $N_C(q)=\{1,5\}$.
  The edge $26$ is then blue, giving the blue cycle $p,2,6,q,3,7,p$.
  All cases lead to a contradiction.
\end{proof}

\begin{proposition}\label{prop:core9}
  Every nonbipartite graph in $\Hc_9$ contains a Hamilton-connected
  induced subgraph on eight vertices. The bipartite graphs in $\Hc_9$
  are precisely $K_{5,5}$ and $K_{5,5}-e$.
\end{proposition}
\begin{proof}
  In the nonbipartite case, apply Lemmas~\ref{lem:red8}, \ref{lem:attach8}
  and~\ref{lem:densecore}. The eight cycle vertices induce a red graph of
  minimum degree at least five, so Lemma~\ref{lem:hc} applies.
  If $H$ is bipartite, a bipartition class of size six would give a blue
  $K_6$. Both classes must therefore have size five. Lemma~\ref{lem:blue5}
  gives the stated alternatives. Conversely, both graphs
  avoid a red $C_9$ and a blue $C_6$.
\end{proof}

\subsection{The two Ramsey equalities}
The following degree bound reduces the two Ramsey problems to the
critical graphs described above.
\begin{lemma}\label{lem:globaldegree}
  Let $n\ge8$. If a graph $G$ of order $2n-1$ contains neither a red $C_n$
  nor a blue $W_7$, then $\delta(G)\le n-3$.
\end{lemma}
\begin{proof}
  Suppose $\delta(G)\ge n-2$.
  The graph $G$ is nonbipartite, since a bipartite graph of this order has
  an independent set of size at least $n\ge8$, giving a blue clique
  containing $W_7$.
  If $G$ is not 2-connected, choose $x$ such that $G-x$ is disconnected.
  Each component has at least $n-2$ vertices, so there are exactly two,
  and their orders are either $n-2,n$ or $n-1,n-1$.
  An $n$-vertex component has minimum degree at least $n-3\ge n/2$,
  giving a red $C_n$.
  In the other case, each component has minimum degree at least $n-3$
  and is Hamilton-connected by Lemma~\ref{lem:hc}.
  Vertex $x$ has two red neighbors in one component, whose spanning path
  closes through $x$ to a red $C_n$.
  Therefore $G$ is 2-connected.

  Since $n-2\ge(2n+1)/3$, the theorem of Brandt et al.
  implies that $G$ is weakly pancyclic with girth three or four.
  Dirac's circumference bound gives $c(G)\ge2n-4\ge n$.
  Hence $G$ contains $C_n$, a contradiction.
\end{proof}

\begin{proof}[Proof of the first two equalities in Theorem~\ref{thm:main}]
  For either $n=8$ or $n=9$, the red graph $2K_{n-1}$ has no $C_n$,
  and its bipartite complement has no $W_7$. Thus $R(C_n,W_7)\ge2n-1$.

  For the upper bound, suppose that $G$ is a graph on $2n-1$ vertices
  with no red $C_n$ and no blue $W_7$.
  The cycle Ramsey formula~\cite{kr} gives $R(C_n,C_6)=n+2$ for these
  two values of $n$. A blue neighborhood cannot contain a blue $C_6$,
  so every blue degree is at most $n+1$. Consequently
  $\delta(G)\ge n-3$. Lemma~\ref{lem:globaldegree} yields equality.
  Choose $v$ of red degree $n-3$ and put $B=N_{\cl G}(v)$.
  Then $|B|=n+1$ and $G[B]\in\Hc_n$.

  For $n=8$, Proposition~\ref{prop:core8} supplies a Hamilton-connected
  seven-vertex core, contradicting Lemma~\ref{lem:core} with $k=3$.
  For $n=9$, the same argument using Proposition~\ref{prop:core9}
  handles a nonbipartite $G[B]$.

  It remains to consider $n=9$ with $G[B]\cong K_{5,5}$ or $K_{5,5}-e$.
  Let its two parts be $U,V$, each of order five, and choose
  $a\in N_G(v)$; this set has order six.
  If $a$ has three blue neighbors in $U$, take one, $u$, as a hub.
  The set $(U\setminus\{u\})\cup\{v\}$ is a blue five-clique,
  and $a$ has at least two blue neighbors in it.
  A spanning blue path between these neighbors closes through $a$ to a
  blue six-cycle, all blue-adjacent to $u$. This is a blue $W_7$.
  Therefore $a$ has at least three red neighbors in $U$, and likewise
  at least three in $V$.

  Choose $u_0\in U$ to be an endpoint of the possible missing red edge,
  or an arbitrary vertex if there is no missing edge.
  Let $U'=U\setminus\{u_0\}$ and choose a four-element set $V'\subset V$
  containing a red neighbor of $a$.
  The red graph between $U'$ and $V'$ is complete, and $a$ has a red
  neighbor in each part. An alternating spanning path in this $K_{4,4}$
  between such neighbors closes through $a$ to a red $C_9$.
  This final contradiction proves both equalities.
\end{proof}

\section{The nine-vertex wheel}\label{sec:wheel-nine}
We now prove $R(C_8,W_9)=15$. The coloring with red graph $2K_7$
gives the lower bound. For the upper bound, suppose that a red--blue
coloring of $K_{15}$ has no red $C_8$ and no blue $W_9$, and let $G$
be its red graph. We first restrict the structure of $G$ and then prove
that its minimum degree is five.

\begin{corollary}\label{w9:cor:no-core}
  A red--blue coloring of $K_{15}$ with no red $C_8$ and no blue $W_9$
  contains no red Hamilton-connected subgraph on seven vertices.
  In particular, it contains no seven-vertex red subgraph of minimum
  degree at least four.
\end{corollary}
\begin{proof}
  Apply Lemma~\ref{lem:core} with $n=8,k=4$.
\end{proof}

\subsection{Basic structural reductions}
\begin{lemma}\label{w9:lem:degree-lower}
  $\delta(G)\ge4$.
\end{lemma}
\begin{proof}
  A vertex with at least eleven blue neighbors would have a red $C_8$
  or a blue $C_8$ in its blue neighborhood, since $R(C_8,C_8)=11$.
  The latter cycle would form a blue $W_9$ with that vertex as hub.
  Thus every blue degree is at most ten, so every red degree is at least
  $14-10=4$.
\end{proof}

\begin{lemma}\label{w9:lem:no-k8}
  Every eight-vertex set spans at least three red edges.
  In particular, there is no blue $K_8$, and $G$ is nonbipartite.
\end{lemma}
\begin{proof}
  Suppose $|Q|=8$ and $G[Q]$ has at most two edges.
  Let $D$ consist of the endpoints of those red edges, and put $U=Q\setminus D$,
  so $|U|\ge4$. Every vertex of $U$ is blue-adjacent to all other vertices of $Q$.
  If an outside vertex $p$ has at least three blue neighbors in $Q$ and one
  of them is $u\in U$, then the blue graph on $Q\setminus\{u\}$ has minimum
  degree at least four and is Hamilton-connected.
  Take a spanning path between two other blue neighbors of $p$, and close
  it through $p$. This is a blue eight-cycle with hub $u$.
  Consequently every outside vertex either has at least six red neighbors
  in $Q$, or is red-complete to $U$.

  There are seven outside vertices. If four have at least six red neighbors
  in $Q$, order them as $p_1,\ldots,p_4$.
  Each pair $p_i,p_{i+1}$, with indices taken modulo four, has at least
  four common red neighbors in $Q$;
  Hall's theorem gives distinct representatives $q_i$ and the red cycle
  $p_1q_1p_2q_2p_3q_3p_4q_4p_1$.
  Otherwise at least four outside vertices are red-complete to $U$,
  giving a red $K_{4,4}$ and hence a red $C_8$.

  A bipartite graph on 15 vertices has an independent set of size at
  least eight, which would be such a blue clique.
\end{proof}

\begin{lemma}\label{w9:lem:small-degree}
  Let $8\le h\le10$. If $H$ has $h$ vertices and $\delta(H)\ge4$,
  then $H$ or $\cl H$ contains $C_8$.
\end{lemma}
\begin{proof}
  Regard $H$ as red and its complement as blue.
  If $H$ is disconnected, each component has at least five vertices,
  so two components give a blue $K_{4,4}$.
  If $H$ is connected but not 2-connected, choose a cut vertex $x$.
  Each component of $H-x$ has at least four vertices, again giving
  a blue $K_{4,4}$.

  Suppose $H$ is 2-connected.
  If it is nonbipartite, then $4\ge(h+2)/3$ and its circumference is at
  least eight. Weak pancyclicity gives a red $C_8$.
  If it is bipartite, both parts have size at least four.
  A part of size four is red-complete to the other part.
  The only other case has parts of size five.
  Choose four vertices from each part; every chosen vertex has at least
  three neighbors in the opposite chosen part.
  The missing cross-edges form a matching. Hence the chosen vertices
  contain a spanning red copy of $K_{4,4}-M$, where $M$ is a perfect
  matching. This graph is Hamiltonian.
\end{proof}

\begin{lemma}\label{w9:lem:connectivity}
  $G$ is 2-connected.
\end{lemma}
\begin{proof}
  Suppose otherwise. Choose $x$ so that $G-x$ is disconnected;
  if $G$ itself is disconnected, any vertex can be chosen, since its
  components have at least five vertices.
  By Lemma~\ref{w9:lem:degree-lower}, every component of $G-x$ has at least
  four vertices.
  Three components would give a blue wheel: choose its hub in one
  component and its rim as a blue $K_{4,4}$ between two others.
  Thus there are exactly two components, with vertex sets $A,B$,
  where $4\le |A|\le |B|$ and $|A|+|B|=14$.

  All edges between $A$ and $B$ are blue.
  Every vertex of $B$ has at most three blue neighbors within $B$,
  since four of them and four vertices of $A$ would form its blue rim.
  Therefore $\delta(G[B])\ge |B|-4$.
  If $|B|\ge8$, then $8\le |B|\le10$ and $\delta(G[B])\ge4$.
  Lemma~\ref{w9:lem:small-degree} gives a red $C_8$ in $B$ or a blue $C_8$
  there, the latter forming a blue wheel with any hub in $A$.

  It follows that $|A|=|B|=7$.
  If either part has blue maximum degree at most two, its red induced
  graph has minimum degree at least four and is Hamilton-connected,
  contrary to Corollary~\ref{w9:cor:no-core}.
  Choose $a\in A$ with three blue neighbors $a_1,a_2,a_3$ in $A$.
  There are at least two blue edges in $B$, since otherwise its red
  graph is $K_7$ or $K_7-e$, again Hamilton-connected.
  Two selected blue edges are either adjacent or disjoint.
  Add vertices to obtain a five-vertex blue linear forest with two edges
  and three path components $P_1,P_2,P_3$; isolated vertices count as paths.
  Traversing
  \[
  P_1,\ a_1,\ P_2,\ a_2,\ P_3,\ a_3,\ P_1
  \]
  gives a blue cycle of length eight, with all connecting edges between
  $A$ and $B$. Every rim vertex is blue-adjacent to $a$.
  This is a blue $W_9$, a contradiction.
\end{proof}

The final step of this proof is related to the disconnected-graph
lemma of Chng et al.~\cite[Lemma~4.4]{trees}.

\begin{corollary}\label{w9:cor:degrees}
  $\delta(G)\in\{4,5\}$.
\end{corollary}
\begin{proof}
  The lower bound follows from Lemma~\ref{w9:lem:degree-lower}.
  If $\delta(G)\ge6$, then $G$ is nonbipartite and 2-connected,
  $6\ge(15+2)/3$, and its circumference is at least twelve.
  Weak pancyclicity gives a red $C_8$.
\end{proof}

\subsection{Hamilton-connected subgraphs in blue neighborhoods}
\begin{lemma}\label{w9:lem:k6}
  A red--blue coloring of $K_{10}$ containing a red $K_6$ contains
  a red $C_8$, a blue $C_8$, or a red Hamilton-connected subgraph
  on seven vertices.
\end{lemma}
\begin{proof}
  Let $X$ be the red six-clique and let $T$ be the set of four other
  vertices. Assume there is no red $C_8$ and no red Hamilton-connected
  subgraph on seven vertices.
  A clique together with one additional vertex having at least three
  neighbors in it is Hamilton-connected.
  Indeed, if the additional vertex is an endpoint, start there and
  traverse the clique to the prescribed other endpoint.
  If both endpoints lie in the clique, choose two of its three neighbors
  whose pair is not the pair of prescribed endpoints; a spanning path
  in the clique between those endpoints can contain the chosen edge.
  Subdivide that edge by the additional vertex.
  Thus every vertex of $T$ has at most two red neighbors in $X$.

  If two vertices of $T$ have distinct two-element red-neighbor sets,
  regard those sets as distinct edges of $K_6$.
  These two edges form a linear forest and lie in a Hamilton cycle of
  $K_6$. Subdividing them by the corresponding outside vertices gives
  a red $C_8$.
  Hence all two-element red-neighbor sets, if present, equal one pair $D$.

  Order $T$ cyclically so that two vertices with neighbor set $D$
  are consecutive if there are at least two such vertices.
  Otherwise choose two vertices with at most one red neighbor to be
  consecutive. Let $H$ be the red graph of the coloring and, for this
  order, put
  \[
  S_i=X\cap N_{\cl H}(t_i)\cap N_{\cl H}(t_{i+1}),\qquad i\pmod4,
  \]
  Every $S_i$ has size at least three, and the specially chosen pair
  gives one set of size at least four. Hall's condition holds:
  nonempty subfamilies of at most three sets have union of size at least three,
  and the full family has union of size at least four.
  Distinct representatives $x_i\in S_i$ give the blue cycle
  $t_1x_1t_2x_2t_3x_3t_4x_4t_1$.
\end{proof}

\begin{lemma}\label{w9:lem:three-attachments}
  Let $X,Y$ be disjoint vertex sets of sizes seven and three in a
  red--blue coloring of a complete graph. Suppose each vertex of $Y$ has at most one blue neighbor
  in $X$, and there is no red $C_8$.
  Then the red graph on $X$ has at most two edges.
\end{lemma}
\begin{proof}
  We first show that any pair $F$ of red edges in $X$ forces all three vertices
  of $Y$ to have the same unique blue neighbor $s$ in $X$, and that $s$
  has degree one in $F$.

  The two edges form $P_3$ or $2K_2$. Add isolated vertices of $X$ to obtain
  a five-vertex linear forest with three path components.
  Orient and order those components cyclically.
  Between the end of each component and the start of the next is a
  two-element joining pair $E_i$, for $i=1,2,3$.
  The three pairs have empty common intersection and any two intersect
  in at most one vertex. A leaf of $F$ occurs in one pair; an added
  isolated vertex occurs in two.
  Matching the pairs to distinct vertices of $Y$ red-adjacent to both
  members produces a red eight-cycle.
  Call the unique blue neighbor of a vertex of $Y$, when it exists,
  its exception.

  If the exceptions that exist are pairwise distinct, Hall's condition
  holds. Each pair has at least one eligible vertex of $Y$.
  Two pairs cannot have just one eligible vertex between them: the other
  two vertices of $Y$ would have the same exception in their intersection.
  All three pairs together have all of $Y$ eligible, since no exception
  belongs to every pair.

  Suppose exactly two vertices of $Y$ share an exception $s$, and the third
  has a different exception $t$ or none. Choose the added isolated vertices
  outside $\{s,t\}$ (omit $t$ when absent). There are enough choices both
  for $P_3$ and for $2K_2$. Orient the components so that no joining pair
  is $\{s,t\}$. For $P_3$ this is automatic because every joining pair
  contains an added isolated vertex; for $2K_2$ one can choose which
  endpoints of its two edges form the only pair without an isolated vertex.
  At most one pair contains $s$. Assign the third vertex of $Y$ to that
  pair, and the two remaining vertices to the other pairs.
  If no pair contains $s$, assign the third vertex to any pair avoiding
  its exception and complete the assignment.

  Finally suppose all three exceptions equal $s$.
  If $s$ is not a leaf of $F$, choose added isolated vertices different
  from $s$. Then none of the joining endpoints equals $s$, and every
  assignment works. This proves the asserted restriction on $F$.

  If the red graph on $X$ had at least three edges, the restriction for one pair
  would fix a common exception $s$ for all of $Y$.
  Among three edges, some two are both incident with $s$ or both avoid $s$.
  Then $s$ has degree two or zero in that pair, a contradiction.
\end{proof}

\begin{lemma}\label{w9:lem:no-blue-core}
  If $v$ has red degree four, its blue neighborhood contains no blue
  Hamilton-connected subgraph of order seven.
\end{lemma}
\begin{proof}
  Put $B=N_{\cl G}(v)$, so $|B|=10$.
  Suppose $X\subseteq B$ induces a blue Hamilton-connected graph of order
  seven, and set $Y=B\setminus X$.
  Every $y\in Y$ has at most one blue neighbor in $X$; otherwise a blue
  spanning path in $X$ closes through $y$ to an eight-cycle with hub $v$.
  Lemma~\ref{w9:lem:three-attachments} gives $e(G[X])\le2$.
  Since every edge from $v$ to $X$ is blue, the eight-vertex set $X\cup\{v\}$
  contradicts Lemma~\ref{w9:lem:no-k8}.
\end{proof}

\begin{corollary}\label{w9:cor:neighborhood}
  If $v$ has red degree four, then its ten-vertex blue neighborhood
  contains no Hamilton-connected seven-vertex subgraph in either color,
  and no six-clique in either color.
  Both color graphs on this neighborhood have no cycle longer than seven.
\end{corollary}
\begin{proof}
  Write $B=N_{\cl G}(v)$. A blue $C_8$ in $B$ would form a blue wheel with hub $v$,
  and a red $C_8$ is forbidden globally.
  The Hamilton-connected subgraphs are excluded by Corollary~\ref{w9:cor:no-core} and
  Lemma~\ref{w9:lem:no-blue-core}.
  Applying Lemma~\ref{w9:lem:k6} in each color excludes both six-cliques.

  A monochromatic nine-cycle gives a monochromatic eight-cycle by
  \cite[Lemma~2.1(1)]{kr}; a monochromatic ten-cycle does so by
  \cite[Lemma~2.1(2)]{kr}. Thus neither color has a cycle longer than seven.
\end{proof}

\subsection{The ten-vertex local dichotomy}
We prove a structural statement for arbitrary red--blue colorings of
$K_{10}$, using the published stability analysis of
F\"uredi et al.~\cite{fkv}. For reproducibility, the source theorem,
corollary, section, and case numbers below refer to the fixed version
arXiv:1507.05338v2. We state the structural alternatives used here
and give the edge counts and complement arguments needed for this
application.

\begin{lemma}\label{w9:lem:dense7}
  If $|F|=7$, $e(F)\ge17$, and $\delta(F)\ge3$, then $F$ is Hamilton-connected.
  Consequently, a seven-vertex graph of minimum degree at least two and at
  least seventeen edges is Hamilton-connected or contains $K_6$.
\end{lemma}
\begin{proof}
  Take the Hamilton-connected $8$-closure: repeatedly add a missing edge
  whose endpoints have degree sum at least eight. This preserves
  Hamilton-connectedness in both directions. For completeness, this closure
  rule follows from the usual Hamiltonian closure rule after adjoining a
  new degree-two vertex adjacent to the two prescribed endpoints.

  Suppose the resulting closure is not complete, and let $r$ be its number
  of nonuniversal vertices. Every such vertex has degree at most four,
  since its nonneighbor has degree at least three. Hence
  \[
  34\le \sum d(x)\le6(7-r)+4r=42-2r,
  \]
  so $r\le4$. Each nonuniversal vertex is adjacent to the $7-r$ universal
  vertices. Thus $r\le2$ is impossible, while $r=3$ makes the endpoints
  of any missing edge have degree at least four, again impossible in a
  closed graph. For $r=4$, equality in the degree sum forces all four
  nonuniversal vertices to have degree four, giving the same contradiction.
  The closure is complete.

  For the final assertion, if a vertex has degree two, deleting it leaves
  at least fifteen edges on six vertices and hence a $K_6$. Otherwise the
  first assertion applies.
\end{proof}

\paragraph{Kopylov's theorem.}
Put
\[
h(n,k,a)=\binom{k-a}{2}+a(n-k+a).
\]
For integers $n\ge k$ and $1\le a<k/2$, the graph $H_{n,k,a}$
has a partition $(A,B,C)$ of sizes
$a,n-k+a,k-2a$, respectively, with all edges in $A\cup C$ and all
edges between $A$ and $B$, and no other edges.

We use the following formulation of Kopylov's theorem~\cite{kopylov1977};
see also~\cite[Theorem~1.3]{fkv}.
\begin{theorem}
  \label{w9:thm:kopylov}
  Let $n\ge k\ge5$. If $F$ is a 2-connected graph of order $n$ and
  circumference less than $k$, then
  \[
  e(F)\le\max\{h(n,k,2),h(n,k,\lfloor(k-1)/2\rfloor)\},
  \]
  with equality only if $F$ is isomorphic to $H_{n,k,2}$ or
  $H_{n,k,\lfloor(k-1)/2\rfloor}$.
\end{theorem}
In particular, for $k=8$ the pairs of bounds at $n=8,9,10$ are
$(19,19),(21,22),(23,25)$.

\begin{lemma}\label{w9:lem:blocks}
  Suppose a coloring of $K_{10}$ contains no monochromatic $C_8$, no
  monochromatic $K_6$, and no monochromatic Hamilton-connected subgraph
  on seven vertices. Then each color graph
  having at least twenty-three edges is 2-connected.
\end{lemma}
\begin{proof}
  Let $F$ be such a color graph, say red, and suppose it is not
  2-connected. By \cite[Lemma~2.1(1)--(2)]{kr}, a monochromatic cycle
  of length nine or ten would give a monochromatic $C_8$.
  Thus $F$ has circumference at most seven.
  Let $B$ be a largest block and set $b=|B|$.
  Bridges are allowed as two-vertex blocks. Every vertex outside $B$ has
  at most one red neighbor in $B$, since two such edges would enlarge the block.

  If $b\le4$, the block edge bound
  \[
  e(F)\le\sum_{D\text{ block}}2(|D|-1)\le18
  \]
  is a contradiction.

  If $b=6$, choose four outside vertices $t_1,\ldots,t_4$. Every
  consecutive pair has at least four common blue neighbors in $B$.
  Distinct representatives from these four sets give a blue $C_8$.
  If $b=5$, the corresponding sets each have size at least three.
  Hall can fail only if their full union has size three. In that event
  all four sets equal $B\setminus\{x,y\}$ for the same pair $x,y$;
  the unique red attachments of the outside vertices alternate
  $x,y,x,y$ in their cyclic order. Reorder them as $x,x,y,y$ instead.
  Now a consecutive equal-attachment pair has four common blue neighbors,
  and Hall holds. This again gives a blue $C_8$.

  For $b\ge7$, put $r=10-b$. Contracting $B$ to one vertex creates
  no parallel edges, so the number of edges not internal to $B$ is at
  most $\binom{r+1}{2}$. If $b=7$, this gives $e(F[B])\ge17$.
  Since $B$ is 2-connected it has minimum degree at least two, and
  Lemma~\ref{w9:lem:dense7} shows that $F[B]$ is Hamilton-connected or
  contains $K_6$, a contradiction.

  If $b=8$, then $e(F[B])\ge20$, contradicting the bound 19 in
  Theorem~\ref{w9:thm:kopylov}.
  Finally, if $b=9$, then $e(F[B])\ge22$, so equality holds in
  Kopylov's bound. Since $h(9,8,2)=21<22=h(9,8,3)$, uniqueness in
  the equality statement gives
  \[
  F[B]\cong H_{9,8,3}=K_3\vee(K_2+\overline K_4).
  \]
  Its blue graph contains a six-vertex $K_6-e$. The remaining vertex
  of $F$ has at most one red neighbor in $B$, so it has at least five
  blue neighbors in this six-vertex set. The resulting seven-vertex
  blue graph has minimum degree at least four and hence is Hamilton-connected,
  a contradiction.
\end{proof}

\paragraph{The structural input at equality.}
The next lemma records the precise consequence of the proof of
\cite[Theorem~5.1]{fkv} needed at the threshold $e(F)=23$.
Write $\mathcal G_i=\mathcal G_i(10,8)$ for the families defined
in Sections 4 and 5 of that paper.

\begin{lemma}\label{w9:lem:threshold-structure}
  Let $F$ be a $2$-connected graph with $|F|=10$, $c(F)=7$, and
  $e(F)\ge23$. Then at least one of the following holds:
  \begin{enumerate}
    \item[(i)] $F$ is a spanning subgraph of a member of
    $\mathcal G_i$ for some $i\in\{1,2,3,5,6,7,8\}$;
    \item[(ii)] $F$ has a seven-cycle $C$ such that the three vertices
    outside $C$ are independent and each has degree two in $F$.
    In this case $e(F)=23$ and $e(F[V(C)])=17$.
  \end{enumerate}
\end{lemma}
\begin{proof}
  Choose a longest cycle $C$ maximizing
  $\sum_{x\in V(C)}d_F(x)$. Following~\cite{fkv}, a bridge is a
  component of $F-V(C)$, and its attachments are its neighbors on $C$.
  Claims 5.2--5.7 of that paper apply to this choice of $C$: their
  hypotheses are $2$-connectivity and the maximality conditions on
  $C$, with $c(F)=7$. They do not require a strict edge inequality.
  In particular, every nonsingleton bridge $S$ has exactly two
  attachments $a,b$ and is a $J_3$-bridge: $F[S]$ is connected,
  $F[S\cup\{a,b\}]+ab$ is $2$-connected, and a longest $a$--$b$
  path with internal vertices in $S$ has three edges.
  The graph $F[S]$ is a star. If $|S|\ge3$, each leaf has exactly
  one attachment neighbor, the same for all leaves, and the center
  has at most two attachment neighbors.

  Case 3 of the proof of~\cite[Theorem~5.1]{fkv} gives the following
  alternatives, considered in the listed order. A family entry means
  containment as a spanning subgraph.
  \begin{center}
  \small
  \renewcommand{\arraystretch}{1.15}
  \begin{tabular}{@{}c p{0.42\textwidth} p{0.41\textwidth}@{}}
  \toprule
  Source case & Bridge configuration & Containing family or alternative\\
  \midrule
  3.1 & A bridge has at least three attachments.
   & $\mathcal G_1,\mathcal G_2,\mathcal G_3$.\\
  3.2 & Two $J_3$-bridges have different attachment pairs.
   & $\mathcal G_5$.\\
  3.3 & There is a $J_3$-bridge, and all such bridges have the same
  attachment pair.
   & $\mathcal G_2,\mathcal G_3,\mathcal G_5,$\newline
     $\mathcal G_6,\mathcal G_7,\mathcal G_8$.\\
  3.4 & Every bridge is a singleton of degree two in $F$.
   & Alternative (ii).\\
  \bottomrule
  \end{tabular}
  \end{center}
  Since $|F-V(C)|=3$, there is at most one nonsingleton bridge;
  hence Case 3.2 cannot occur. If Case 3.1 does not occur, then
  a nonsingleton bridge places $F$ in Case 3.3. Otherwise all
  bridges are singletons, and $2$-connectivity makes each have
  exactly two attachments, giving Case 3.4. Thus the alternatives
  are exhaustive at order ten.

  In Case 3.4, take any $z\notin V(C)$. The graph
  $F[V(C)\cup\{z\}]$ is $2$-connected: it contains the cycle $C$
  and a vertex joined to two distinct vertices of $C$. It has
  order eight and circumference at most seven, so
  Theorem~\ref{w9:thm:kopylov} gives
  \[
    e(F[V(C)])+2\le h(8,8,2)=h(8,8,3)=19.
  \]
  Consequently $e(F)=e(F[V(C)])+6\le23$.
  The hypothesis $e(F)\ge23$ forces equality in both bounds,
  proving the last assertion of (ii).

  The edge-bound alternative in the statement of
  \cite[Theorem~5.1]{fkv} is $e(F)\le h(10,8,2)=23$.
  The bridge analysis above supplies the additional structure
  at equality; it uses no assumption that $e(F)>23$.
\end{proof}

We recall the relevant features of the containing families, as
defined in Sections 4 and 5 of~\cite{fkv}. The edge counts and
Hamilton-connected subgraphs needed here are established in the next
lemma.

\begin{itemize}
  \item $\mathcal G_1(n,8)=\{H_{n,8,3}\}$.
  \item In $\mathcal G_2$, $A$ is a three-clique, $B$ is independent,
  all $A$--$B$ edges occur, and each $z\in J$ has
  $N(z)=\{a_1,b_1\}$ for fixed $a_1\in A,b_1\in B$.
  \item In $\mathcal G_3$, $A$ is a clique of order three, and all
  $A$--$B$ edges are present. The set $J$ has at least two star components of order at least
  two, attached to the same pair $\{a_1,a_2\}\subset A$;
  there are no $B$--$J$ edges.
  \item In $\mathcal G_5,\mathcal G_6,\mathcal G_7,\mathcal G_8$,
  there is a clique $A$ of order $3,4,4,5$, respectively.
  Outside $A$ are $J_3$-bridges and isolated vertices; each isolated
  vertex has exactly two neighbors in $A$. The additional attachment
  restrictions distinguishing these families are not needed here.
\end{itemize}

\begin{lemma}\label{w9:lem:families}
  Let $F$ be a graph of order ten that is a spanning subgraph of a
  member of one of the seven families above. If $c(F)\le7$ and
  $e(F)\ge23$, then $\cl F$ contains a Hamilton-connected subgraph
  on seven vertices.
\end{lemma}
\begin{proof}
  For $\mathcal G_1$, the complement of $H_{10,8,3}$ induces the
  Hamilton-connected graph $K_7-e$ on $B\cup C$. The complement of
  every spanning subgraph of $H_{10,8,3}$ contains this graph.

  For $\mathcal G_2$, write $j=|J|$ and $b=|B|$, so $b+j=7$.
  The defining graph has
  \[
  3+3b+2j=24-j
  \]
  edges. At least twenty-three edges forces $j\le1$. The seven
  vertices $B\cup J$ then induce a blue $K_7$ or $K_7-e$.

  A $J_3$-bridge on $s$ vertices contributes, counting both its
  internal edges and its attachment edges, at most $2s$ edges when
  $s\ge3$, and at most five when $s=2$. An isolated vertex
  contributes two.
  To justify the first count, a bridge with at least three vertices is
  a star. If two leaves attached to different endpoints, or one leaf
  attached to both endpoints, a path between the attachment endpoints
  through two leaves and the center would have four edges.
  This contradicts the defining maximum length three.
  Thus all leaves have one common attachment neighbor, and the center
  has at most two. The bound is $(s-1)+(s-1)+2=2s$.
  For $\mathcal G_5,\mathcal G_6,
  \mathcal G_7,\mathcal G_8$, with $a=|A|$ and $10-a$ outside
  vertices, the number of edges is at most
  \[
  \binom a2+2(10-a)+\left\lfloor\frac{10-a}{2}\right\rfloor.
  \]
  For $a=3,4,5$ these are $20,21,22$, respectively, all too small.

  For $\mathcal G_3$, put $b=|B|$, $s=|J|=7-b$.
  There are at least two nontrivial star components, so $s\ge4$
  and $b\le3$. If $b\le1$, the star forest on $J$ has at most
  $s-2$ edges. Allowing every vertex of $J$ both attachment neighbors gives
  \[
  e(F)\le3+3b+\binom b2+2s+(s-2)=22+\binom b2=22.
  \]
  If $b=2$ or $3$, then $s=5$ or $4$, respectively. Since $J$ has
  at least two nontrivial star components, its maximum degree is at
  most two. The same holds on $B$, which has at most three vertices.
  There are no $B$--$J$ edges, so the red graph on $B\cup J$ has
  maximum degree at most two. Its blue complement has minimum degree
  at least four and is Hamilton-connected. The same degree bound holds
  in the complement of every subgraph of the family graph.
\end{proof}

\begin{theorem}\label{w9:thm:local}
  Every red--blue coloring of $K_{10}$ with no monochromatic $C_8$
  contains a monochromatic Hamilton-connected subgraph on seven vertices.
\end{theorem}
\begin{proof}
  Suppose not. Lemma~\ref{w9:lem:k6}, applied in both colors, also excludes
  monochromatic $K_6$. One color graph $F$ has at least
  $\lceil45/2\rceil=23$ edges, and it is 2-connected by
  Lemma~\ref{w9:lem:blocks}. By \cite[Lemma~2.1(1)--(2)]{kr}, a monochromatic
  cycle of length nine or ten would give a monochromatic $C_8$.
  Thus both color graphs have circumference at most seven.

  If $c(F)\le6$, \cite[Corollary~5.8]{fkv} applies because
  \[
  e(F)\ge23\ge\left\lfloor\frac{5\cdot10-6}{2}\right\rfloor=22.
  \]
  It gives $F\subseteq H_{10,7,3}$, whose complement contains
  a seven-clique, a contradiction.

  Thus $c(F)=7$. Apply Lemma~\ref{w9:lem:threshold-structure}.
  In alternative (i), Lemma~\ref{w9:lem:families} gives the contradiction.
  In alternative (ii), the seven-cycle $C$ satisfies
  $e(F[V(C)])=17$. Since $C$ spans $F[V(C)]$, this seven-vertex graph
  has minimum degree at least two, so
  Lemma~\ref{w9:lem:dense7} shows that it is Hamilton-connected or contains
  $K_6$. Both alternatives contradict our assumptions.
\end{proof}

\begin{corollary}\label{w9:cor:delta-five}
  Every 15-vertex graph with no red $C_8$ and no blue $W_9$ has
  minimum red degree exactly five.
\end{corollary}
\begin{proof}
  Corollary~\ref{w9:cor:degrees} leaves degrees four and five.
  For a vertex of degree four, Corollary~\ref{w9:cor:neighborhood} contradicts
  Theorem~\ref{w9:thm:local} on its ten-vertex blue neighborhood.
\end{proof}

\subsection{Nine-cycles and ten-cycles}
\begin{lemma}\label{w9:lem:nine-degree-five}
  Every graph on fifteen vertices with minimum degree at least five
  that contains $C_9$ also contains $C_8$.
\end{lemma}
\begin{proof}
  Let $G$ be a graph on fifteen vertices with $\delta(G)\ge5$.
  Suppose that $G$ contains no $C_8$, and fix a cycle
  $C=x_0x_1\cdots x_8x_0$. Indices are taken modulo nine, and
  $T=V(G)\setminus V(C)$ has six vertices.
  Every chord of cyclic distance two is absent, since it would
  shorten $C$ to an eight-cycle. Write
  \[
  A_i=x_ix_{i+3},\qquad B_i=x_ix_{i+4}
  \]
  for the eighteen other possible chords.
  In the cycle and path lists below, an index denotes the corresponding
  vertex of $C$.

  Two vertices of $C$ joined by a six-edge path in $G[C]$ cannot have
  a common neighbor in $T$: that path and the two outside edges
  would form an eight-cycle.
  In particular, each $p\in T$ has at most one neighbor in each set
  \[
  X_r=\{x_r,x_{r+3},x_{r+6}\},\qquad r=0,1,2.
  \]
  The degree assumption therefore gives
  \begin{equation}\label{w9:eq:nine-class-sum}
    \sum_{x\in X_r}d_{G[C]}(x)\ge15-6=9.
  \end{equation}

  The following pairs of chords cannot both be present:
  \begin{equation}\label{w9:eq:nine-pairs}
    (A_i,B_{i+1}),\quad (A_i,B_{i+7}),\quad
    (B_i,B_{i+3}),\quad (B_i,B_{i+6}).
  \end{equation}
  Indeed, up to rotation and reflection the corresponding
  eight-cycles are
  \[
  (0,3,2,1,5,6,7,8),\qquad
  (0,3,4,5,6,7,2,1),\qquad
  (0,4,5,6,7,3,2,1).
  \]
  Let $a_r,b_r$ count the present chords $A_i,B_i$ with
  $i\equiv r\pmod3$, respectively. Subscripts on these counts are read
  modulo three. Relations~\eqref{w9:eq:nine-pairs} give $b_r\le1$,
  while \eqref{w9:eq:nine-class-sum} gives
  \begin{equation}\label{w9:eq:nine-chord-count}
    6+2a_r+b_r+b_{r-1}\ge9.
  \end{equation}
  Thus $a_r\ge1$ for every $r$. A present chord $B_j$ excludes
  $A_{j-1}$ and $A_{j+2}$, so the sole remaining chord of that
  residue class, $A_{j+5}$, must be present.

  There cannot be exactly one $B$-chord. If it is $B_j$, then
  $a_{j-1}=1$ and $b_{j-1}=b_{j-2}=0$, contradicting
  \eqref{w9:eq:nine-chord-count}.
  Next suppose that there is no $B$-chord.
  Here each $a_r\ge2$. In each $X_r$, select a two-edge path and
  denote its center by $c_r$.
  If two centers have cyclic distance two, rotate and reflect the labels so
  that they are $x_0,x_7$. Their paths and cycle edges yield
  \[
  (0,6,5,4,7,1,2,3).
  \]
  Otherwise rotate one center to $x_0$.
  The other two centers are chosen from $\{x_1,x_4\}$ and
  $\{x_5,x_8\}$, respectively, and the pair $x_1,x_8$ is excluded.
  The remaining triples are rotations of $\{x_0,x_1,x_5\}$.
  For this triple of centers the selected edges give the eight-cycle
  \[
  (0,3,4,1,7,6,5,8).
  \]
  Thus there are two or three $B$-chords.

  No two of their indices have cyclic distance two.
  Otherwise take them to be $B_0,B_2$.
  Then $A_5$ is forced, and the eight-cycle
  \[
  (0,1,2,6,7,8,5,4)
  \]
  gives a contradiction.

  Suppose there are exactly two $B$-chords.
  Their indices have distinct residues modulo three and cyclic
  distance one or four. Up to rotation and reflection, there are
  two cases.

  \emph{Case 1: $B_0,B_1$ are present.}
  Then $A_5,A_6$ are forced, and these are the only $A$-chords
  in their respective residue classes. The remaining possible
  $A$-chords, $A_1,A_4,A_7$, have both endpoints in $X_1$ and do
  not affect the following degrees:
  \[
  (d_{G[C]}(x_0),d_{G[C]}(x_3),d_{G[C]}(x_6))=(4,2,3),
  \quad
  (d_{G[C]}(x_2),d_{G[C]}(x_5),d_{G[C]}(x_8))=(2,4,3).
  \]
  Both triples have degree sum nine. The minimum-degree assumption and
  the fact that every vertex of $T$ has at most one neighbor in each triple
  force $x_3$ to have exactly three neighbors in $T$ and $x_8$
  to have exactly two. They also force every vertex of $T$ to
  have exactly one neighbor in $X_2=\{x_2,x_5,x_8\}$.
  The six-edge paths
  \[
  (3,4,0,6,5,1,2),\qquad (3,4,0,8,7,6,5)
  \]
  show that $x_3$ has no common neighbor in $T$ with $x_2$ or
  $x_5$. Thus all three neighbors of $x_3$ in $T$ must
  be adjacent to $x_8$, a contradiction.

  \emph{Case 2: $B_0,B_4$ are present.}
  Now $A_0,A_5$ are forced. The degrees within $C$ of
  $x_2,x_5,x_8$ are exactly $2,3,4$, independently of which
  of $A_1,A_4,A_7$ are present. Each vertex of $T$ has at most
  one neighbor in $X_2$, and the minimum-degree assumption
  requires at least six edges from $X_2$ to $T$.
  It follows that every vertex of $T$ has exactly one neighbor
  in $X_2$, and that $x_8$ has exactly one neighbor in $T$.
  The vertex $x_6$ has degree two in $G[C]$ and therefore at
  least three neighbors in $T$. The six-edge paths
  \[
  (6,5,8,4,0,3,2),\qquad (6,7,8,0,3,4,5)
  \]
  show that $x_6$ has no common neighbor in $T$ with $x_2$ or
  $x_5$. Hence at least three neighbors of $x_6$ in $T$ must
  be adjacent to $x_8$, a contradiction.

  Finally, suppose there are three $B$-chords.
  Their indices have distinct residues and no pair has cyclic
  distance two. The same three-position argument used for the
  centers above shows that their index set is a rotation of
  $\{0,1,5\}$. The forced $A$-chords are then exactly
  $A_5,A_6,A_1$.
  The vertex $x_3$ has degree two in $G[C]$ and hence at least
  three neighbors in $T$. None can be adjacent to any vertex
  of $X_2=\{x_2,x_5,x_8\}$, by the six-edge paths
  \[
  (3,4,0,6,5,1,2),\quad
  (3,4,0,8,7,6,5),\quad
  (3,2,1,4,0,5,8).
  \]
  At most three vertices of $T$ remain, each with at
  most one neighbor in $X_2$.
  However, the degrees in $G[C]$ of the vertices in $X_2$ are $2,5,3$,
  so the minimum-degree assumption requires at least five edges
  from $X_2$ to $T$. This final contradiction completes the proof.
\end{proof}

\begin{lemma}\label{w9:lem:short-long-cycle}
  $G$ contains a red $C_9$ or a red $C_{10}$.
\end{lemma}
\begin{proof}
  Dirac's circumference theorem guarantees a red cycle of length at
  least eight. Among such cycles, choose one of minimum length $\ell$.
  Since red $C_8$ is forbidden, $9\le\ell\le15$.
  Suppose $\ell\ge11$.

  For $\ell\ge13$, any chord of this cycle would give a shorter red
  cycle of length at least eight, using the longer arc.
  There are consequently no red chords.
  Each cycle vertex has at least ten blue neighbors on the cycle.
  Any eight of these neighbors induce a blue graph of minimum degree at
  least five, which is Hamiltonian by Dirac's theorem.
  Together with the chosen vertex, this gives a blue $W_9$.

  If $\ell=12$, the only possible red chords join antipodal vertices.
  Thus the red induced graph on the cycle has maximum degree at most
  three. Each cycle vertex has at least eight blue neighbors on it,
  and any eight chosen neighbors induce a blue graph of minimum
  degree at least four. Dirac's theorem again gives a blue $W_9$.

  Let $\ell=11$ and choose a vertex $p$ outside the cycle.
  It cannot be red-adjacent to two cycle vertices at cyclic distance
  three, four or five: the longer arc together with $p$ would give
  a red cycle of length ten, nine or eight, respectively.
  Thus its red neighbors on the cycle lie in three consecutive
  positions. To see this when the set is nonempty, label one red neighbor
  zero. Every other red neighbor lies among $-2,-1,1,2$; avoiding pairs
  at cyclic distance three or four confines the set to three consecutive
  positions.
  There are therefore eight consecutive cycle vertices blue-adjacent
  to $p$. Label them $0,1,\ldots,7$ in their order on the eleven-cycle.
  All chords of cyclic distance two, three or four are blue, since
  a red one would create a shorter cycle of length at least eight.
  The cycle
  \[
  0,2,4,6,3,1,5,7,0
  \]
  uses only such blue chords and has blue hub $p$.
  Every case $\ell\ge11$ is impossible.
\end{proof}

\begin{lemma}\label{w9:lem:forced-nine-cycle}
  The graph $G$ contains a red $C_9$.
\end{lemma}
\begin{proof}
  By Lemma~\ref{w9:lem:short-long-cycle}, it suffices to rule out the
  case in which $G$ contains a red ten-cycle but no red nine-cycle.
  Write $C=x_0x_1\cdots x_9x_0$ for such a ten-cycle, with subscripts
  taken modulo ten, and put $T=V(G)\setminus V(C)$, so $|T|=5$.
  Every chord at cyclic distance two or three is blue, since a red
  one would give a red nine-cycle or eight-cycle, respectively.

  For $p\in T$, two red neighbors at cyclic distance three or four
  would give a red nine-cycle or eight-cycle by using the longer arc
  and the two edges through $p$. Thus the cyclic distance between
  two members of $N_G(p)\cap V(C)$ belongs to $\{1,2,5\}$.
  Such a set either has at most two vertices or consists of three
  consecutive vertices. Indeed, a set containing an antipodal pair cannot
  contain a third vertex, while a set without an antipodal pair is
  confined to three consecutive positions.

  Suppose first that $p$ has at most two red neighbors on $C$.
  If they lie in $\{x_0,x_2\}$, the blue cycle
  \[
  x_1x_3x_5x_8x_6x_4x_7x_9x_1
  \]
  has hub $p$. Rotation covers every pair at distance two and every
  set of size at most one.
  If the red-neighbor set is $\{x_0,x_1\}$, the chord $x_4x_8$ must be blue:
  otherwise
  \[
  p,x_0,x_9,x_8,x_4,x_3,x_2,x_1,p
  \]
  is a red eight-cycle. But then
  \[
  x_8,x_4,x_7,x_9,x_2,x_5,x_3,x_6,x_8
  \]
  is a blue eight-cycle with hub $p$.
  Finally, if the red-neighbor set is $\{x_0,x_5\}$, both $x_2x_8$ and
  $x_3x_7$ must be blue. A red one would give, respectively,
  \[
  p,x_0,x_1,x_2,x_8,x_7,x_6,x_5,p
  \quad\hbox{or}\quad
  p,x_0,x_1,x_2,x_3,x_7,x_6,x_5,p.
  \]
  The blue cycle
  \[
  x_2,x_8,x_1,x_3,x_7,x_4,x_6,x_9,x_2
  \]
  then has hub $p$.
  Consequently every vertex of $T$ has exactly three consecutive
  red neighbors on $C$.

  Suppose $p,q\in T$ are blue-adjacent, and rotate the labels so that
  $N_G(p)\cap V(C)=\{x_0,x_1,x_2\}$.
  Consider the graph on $x_3,\ldots,x_9$ whose edges are the blue
  chords of cyclic distance two or three. It has the spanning path
  \[
  x_3,x_5,x_8,x_6,x_4,x_7,x_9.
  \]
  In this graph, $x_3,x_8$ have identical neighborhoods, as do
  $x_4,x_9$. Swapping either pair therefore gives a spanning path
  between any prescribed vertex of $\{x_3,x_8\}$ and any prescribed
  vertex of $\{x_4,x_9\}$.
  The three consecutive red neighbors of $q$ cannot contain either
  antipodal pair. Choose a blue neighbor of $q$ in each pair and use
  the corresponding spanning path. Closing it through $q$ gives a
  blue eight-cycle with hub $p$, a contradiction.
  It follows that $G[T]=K_5$.

  Write the red-neighbor triple of $p\in T$ as
  $\{x_{c(p)-1},x_{c(p)},x_{c(p)+1}\}$.
  If two centers have cyclic distance at least three, the two triples
  contain antipodal vertices $a,b$, one in each triple.
  The red edge between the corresponding outside vertices, their
  edges to $a,b$, and a five-edge arc from $a$ to $b$ give a red
  eight-cycle. Thus all centers have pairwise cyclic distance at most
  two. They lie in three consecutive positions, and their red-neighbor
  triples have a common vertex $x$.
  Write these three possible center positions as $c-1,c,c+1$,
  so we may take $x=x_c$. Let $L,M,R$ count the outside vertices
  with these respective centers. Then $L+M+R=5$.
  The vertices $x_{c-1},x_{c+1}$ have, respectively, at least
  $L+M$ and $R+M$ red neighbors in $T$. One of these two numbers
  is at least three, since their sum is $5+M$.
  Thus one of these cycle vertices has at least three neighbors in
  the six-clique $T\cup\{x_c\}$.
  By the clique extension argument in Lemma~\ref{w9:lem:k6}, the
  resulting seven-vertex graph is Hamilton-connected, contrary to
  Corollary~\ref{w9:cor:no-core}.
\end{proof}

\begin{proof}[Proof of the third equality in Theorem~\ref{thm:main}]
  The lower bound was noted at the start of this section.
  In a hypothetical counterexample, Corollary~\ref{w9:cor:delta-five}
  gives $\delta(G)=5$, while Lemma~\ref{w9:lem:forced-nine-cycle}
  gives a red nine-cycle.
  Lemma~\ref{w9:lem:nine-degree-five} then yields a red eight-cycle,
  a contradiction. Therefore $R(C_8,W_9)=15$.
\end{proof}

\section{Concluding remarks}
Theorem~\ref{thm:eleven} extends the established formula for the
eleven-vertex wheel from $n\ge16$ to $n\ge14$.
In both boundary cases, we combine a lower bound on $\kappa(G)$
with a small vertex cover of $G[V(C),X]$, where $C$ lies in the
blue neighborhood of a vertex $v$ of minimum red degree and
$X=N_G(v)$. The resulting vertex cut of order nine separates two vertex sets,
and every vertex has at most four blue neighbors within its own set.
These degree bounds yield paths with suitable endpoints, and
matchings to the cut allow us to join them into a red cycle of
the required order. For the exceptional graphs of order ten that
satisfy Ore's degree-sum condition but are not Hamilton-connected,
we use explicit path constructions and the exclusion of a blue wheel.

The same core lemma is also used to prove the three smaller values
in Theorem~\ref{thm:main}. For $W_7$, the structure of critical
cycle colorings supplies the required Hamilton-connected cores.
For $W_9$, a local lemma and cycle-shortening arguments complete
the proof.
The complete classifications in Appendix~\ref{sec:classification}
explain the previously reported counts of critical colorings.

The present arguments do not settle $R(C_n,W_{11})$ for
$10\le n\le13$. The parameter range of Lemma~\ref{lem:core} already
includes all four cases, since $k=5$ requires only $n\ge10$.
The remaining difficulty is to force a suitable core or to construct
paths of the required orders with prescribed endpoints. For a hypothetical counterexample $G$ of order
$2n-1$ with no red $C_n$ and no blue $W_{11}$, the elementary
blue-neighborhood bound is
\[
\delta(G)\ge2n-1-R(C_n,C_{10}),
\]
which, by the cycle Ramsey formula~\cite{FSR}, gives the lower bounds
$5,2,7,6$ at $n=10,11,12,13$, respectively. These bounds alone do
not supply the connectivity and path structure used above.

One direction is to describe more precisely the graphs in blue
neighborhoods that avoid a blue $C_{10}$ and the relevant red cycle,
as the critical-coloring analysis does for $W_7$. Such descriptions
could yield stronger degree bounds or restrict the possible red
attachments. A second direction is to determine which pairs of core
vertices can be joined by a path of the required order, and relate
the exceptional pairs to the outside attachments. This could allow
a version of the core argument using a specified set of endpoint
pairs instead of full Hamilton-connectedness. These are possible
extensions of the method; the corresponding structural statements
for the four remaining parameters are not proved here.

\section*{Acknowledgements}
The authors used OpenAI's Codex to assist with literature searches,
to offer methodological suggestions for mathematical proofs, to
review proof logic,  and to revise the presentation. All conclusions
and mathematical proofs have been checked by the authors, who take full responsibility for the mathematical claims, proofs, references,
and final manuscript.

\appendix
\section{The complete critical classifications}\label{sec:classification}
We strengthen the core descriptions to complete classifications of
$\Hc_8$ and $\Hc_9$, recovering the counts structurally.
\begin{lemma}\label{lem:twoedges}
  Let $X$ be a vertex set of size at least seven, and let $p,q$
  be distinct vertices outside $X$.
  Suppose there is no blue $C_6$ and each of $p,q$ has at most one
  red neighbor in $X$. Then there are at most two blue edges within $X$.
  If there are two, $p,q$ have the same unique red neighbor in $X$,
  which has degree one in the graph formed by those two edges.
\end{lemma}
\begin{proof}
  For a blue edge $ab$ in $X$, both orientations $p,a,b,q$ and $p,b,a,q$
  fail to be blue paths if and only if $p,q$ have the same unique red
  neighbor in $\{a,b\}$.
  Two disjoint blue edges, each admitting an orientation, would therefore
  give a blue six-cycle through $p,q$.
  For two adjacent blue edges $ab,bc$, choose a common blue neighbor
  $r\in X\setminus\{a,b,c\}$ of $p,q$.
  The cycles $p,a,b,c,q,r,p$ and $q,a,b,c,p,r,q$ can both fail only if
  $p,q$ have the same unique red neighbor, equal to $a$ or $c$.
  Thus any pair of blue edges forces a common red attachment $s$ having
  degree one in that pair.
  If three blue edges exist, two are both incident with $s$ or both avoid
  $s$. The degree of $s$ in this pair is respectively two or zero,
  a contradiction.
\end{proof}

For $r\in\{7,8\}$, define a family $\mathcal F_r$ as follows.
Take a set $X$ of size $r$ and two additional vertices $y,z$, and put
$H[X]=K_r-F$, where $F$ is the graph of blue edges on $X$.
Each of $y,z$ has at most one red neighbor in $X$.
The symbol $\bot$ denotes the absence of a red neighbor in $X$;
an attachment $a$ denotes the unique red neighbor $a$.
Let $\varepsilon=1$ when $yz$ is red, and $\varepsilon=0$ otherwise.
When $F$ is the single edge $ab$, write $D=\{a,b\}$ and $O=X\setminus D$.
All vertices of $X$ not displayed in $F$ are isolated in $F$.
The permitted choices, up to isomorphism, are given in Table~\ref{tab:families}.

\begin{table}[htbp]
  \centering\small
  \renewcommand{\arraystretch}{1.25}
  \begin{tabular}{@{}lp{76mm}cc@{}}
    \toprule
    $F$ & Attachments of $y,z$ & $\varepsilon$ & Number\\
    \midrule
    $\varnothing$ & $(\bot,\bot),(a,\bot),(a,a)$ & $0,1$ & 6\\
    $\varnothing$ & $(a,b)$, $a\ne b$ & 0 & 1\\
    $ab$ & $(\bot,\bot)$ & $0,1$ & 2\\
    $ab$ & $(a',\bot)$; $a'\in D$ or $a'\in O$ & $0,1$ & 4\\
    $ab$ & $(a',a')$; $a'\in D$ or $a'\in O$ & $0,1$ & 4\\
    $ab$ & distinct attachments, of types $DD,DO,OO$ & 0 & 3\\
    $P_3$ & $(a',a')$, with $d_F(a')=1$ & $0,1$ & 2\\
    $2K_2$ & $(a',a')$, with $d_F(a')=1$ & $0,1$ & 2\\
    \midrule
    & Total & & 24\\
    \bottomrule
  \end{tabular}
  \caption{The family $\mathcal F_r$, for $r=7$ and $r=8$.}
  \label{tab:families}
\par\smallskip
  \begin{minipage}{\linewidth}
  \footnotesize
  \textit{Reading the table.} An attachment pair lists the unique red
  core neighbors of $y,z$; $\bot$ means that no such neighbor exists.
  For example, $(a,\bot)$ means that $ya$ is the only red edge from
  $y$ to $X$, while every $z$--$X$ edge is blue. When $F=ab$,
  $D=\{a,b\}$ and $O=X\setminus D$; $DD$, $DO$, and $OO$ specify
  whether the two distinct attachments lie in $D$, one in each set,
  or both in $O$. For example, $DO$ is represented by $(a,c)$ with
  $c\in O$. The column $\varepsilon$ records the color of $yz$
  ($1$ for red, $0$ for blue); the last column counts isomorphism
  classes, including both choices when $\varepsilon=0,1$ is listed.
  \end{minipage}
\end{table}

\begin{theorem}\label{thm:classification}
  Up to isomorphism,
  \[
  \Hc_8=\mathcal F_7,\qquad
  \Hc_9=\mathcal F_8\cup\{K_{5,5},K_{5,5}-e\}.
  \]
  In particular, $|\Hc_8|=24$ and $|\Hc_9|=26$.
\end{theorem}
\begin{proof}
  For $H\in\Hc_8$, Lemmas~\ref{lem:red7} and~\ref{lem:attach7} give
  a seven-vertex set $X$ such that each outside vertex has at most one
  red neighbor in $X$.
  For nonbipartite $H\in\Hc_9$, use Lemmas~\ref{lem:red8}
  and~\ref{lem:attach8} with $|X|=8$.
  Lemma~\ref{lem:twoedges} shows that $F=\cl H[X]$ has at most two edges.
  If it has two, the common attachment must be a degree-one vertex of
  $F$, exactly as in the last two rows of the table.
  If $F$ has at most one edge, the first six rows list all attachment
  orbits. Distinct attachments together with a red edge $yz$ are forbidden:
  on any $r-1$ core vertices containing both attachments, the red graph is
  $K_{r-1}$ with at most one edge removed and is Hamilton-connected by
  Lemma~\ref{lem:hc}. A spanning path between the attachments, closed
  through $y,z$, gives a red
  $C_{r+1}$. This proves necessity. The bipartite exceptions follow from
  Proposition~\ref{prop:core9}.

  Conversely, every graph in $\mathcal F_r$ avoids $C_{r+1}$.
  When $yz$ is blue, both outside vertices have red degree at most one.
  When $yz$ is red, they either form a component or a pendant path,
  or are both attached to the same core vertex and form a pendant triangle.
  Thus every red cycle has at most $r$ vertices.
  A blue $C_6$ must use at least two edges within $X$, since only two
  vertices lie outside $X$.
  This is impossible for $|E(F)|\le1$.
  For $|E(F)|=2$, the common attachment has total blue degree one,
  so it cannot lie on a blue cycle. Removing it leaves only one blue edge
  within $X$, again excluding $C_6$.

  Finally, every core vertex has red degree at least $r-3\ge4$, whereas
  the outside vertices have red degree at most two. Thus $X$ is
  determined by the red degrees, and every isomorphism preserves it.
  The four forms of $F$ are distinct.
  For $F=\varnothing$ the attachment types give seven orbits;
  for one edge they give $2+4+4+3=13$; for each two-edge form the
  degree-one vertices form a single orbit and the two choices of
  $\varepsilon$ give two further graphs. The two bipartite exceptions
  are mutually nonisomorphic and distinct from these nonbipartite graphs.
  The stated counts follow.
\end{proof}

\end{document}